\documentclass[a4paper, 11pt]{article}

\usepackage{amssymb, amsmath, mathrsfs, amsthm}%, wasysym}
\usepackage{color}
\usepackage{multicol}
\usepackage{hyperref}
\usepackage[top=3cm, bottom=3cm, left=3cm, right=3cm]{geometry}
\usepackage{float, caption, subcaption}
\usepackage{tikz}
\usepackage{algorithm}
\usepackage{algorithmicx}
\usepackage{algpseudocode}  % ÕâžöºÍalgorithmic²»ŒæÈÝ£¬ÓÃÁËŸÍÒª±šŽí£¬ºÃ¶àÄªÃûÆäÃîµÄŽíÎó£¡£¡£¡£¡£¡

\usepackage{amsfonts}
\usepackage{amscd,color}
\usepackage{amsmath,amsfonts,amssymb,amscd}
\usepackage{indentfirst,graphicx,epsfig}
\input{epsf}
\usepackage{graphicx}
\usepackage{epstopdf}
\usepackage{caption}

\input{epsf}

\newcommand{\bd}{\begin{description}}
\newcommand{\ed}{\end{description}}
\newcommand{\bi}{\begin{itemize}}
\newcommand{\ei}{\end{itemize}}
\newcommand{\be}{\begin{enumerate}}
\newcommand{\ee}{\end{enumerate}}
\newcommand{\beq}{\begin{equation}}
\newcommand{\eeq}{\end{equation}}
\newcommand{\beqs}{\begin{eqnarray*}}
\newcommand{\eeqs}{\end{eqnarray*}}

\definecolor{DarkGreen}{rgb}{0.2, 0.6, 0.3}

\newtheorem{theorem}{Theorem}[section]
\newtheorem{conjecture}{Conjecture}

\newtheorem{lemma}{Lemma}[section]
\newtheorem{definition}{Definition}

\newtheorem{case}{Case}

\newtheorem{claim}{Claim}
\newtheorem{remark}{Remark}[section]
\newtheorem{fact}{Fact}

\begin{document}
\title{\textbf{Ramsey and Gallai-Ramsey numbers for linear forests and kipas}\footnote{Supported by the National Natural Science Foundation of China (Nos. 12201375, 12061059).} }

\author{Ping
Li\footnote{Corresponding author: School of Mathematics and
Statistics, Shaanxi Normal University, Xi'an, Shaanxi, China. {\tt
lp-math@snnu.edu.cn}}, \ \  Yaping Mao \footnote{JSPS International
Research Fellow: Faculty of Environment and Information Sciences,
Yokohama National University, 79-2 Tokiwadai, Hodogaya-ku, Yokohama
240-8501, Japan. {\tt mao-yaping-ht@ynu.ac.jp} }, \ \ Ingo Schiermeyer \footnote{Institut f\"{u}r Diskrete Mathematik und Algebra, Technische Universit\"{a}t Bergakademie Freiberg, 09596 Freiberg, Germany. {\tt Ingo.Schiermeyer@tu-freiberg.de} },\ \ Yifan Yao \footnote{
        School of Mathematics and Statistics, Qinghai Normal University,
        Xining, Qinghai 810008, China.
        {\tt yaoyifan@aliyun.com}}}
\date{}
\maketitle

\begin{abstract}
For two graphs $G,H$, the \emph{Ramsey number} $r(G,H)$ is the minimum integer $n$ such that any red/blue edge-coloring of $K_n$ contains either a red copy of $G$ or a blue copy of $H$. For two graphs $G,H$, the \emph{Gallai-Ramsey number} $\operatorname{gr}_k(G:H)$ is defined as the minimum integer $n$ such that any $k$-edge-coloring of $K_n$ must contain either a rainbow copy of $G$ or a monochromatic copy of $H$. In this paper, the classical Ramsey numbers of linear forest versus kipas are obtained.
We obtain the exact values of $\operatorname{gr}_k(G:H)$, where $H$ is either a path or a kipas and $G\in\{K_{1,3},P_4^+,P_5\}$ and $P_4^+$ is the
graph consisting of $P_4$ with one extra edge incident with inner vertex. \\[0.2cm]
\end{abstract}

\section{Introduction}
All graphs considered in this paper are undirected, finite and
simple. We refer to the book \cite{BondyMurty} for graph theoretical
notation and terminology not described here.
For a graph $G$, let
$V(G)$, $E(G)$, $|G|$, $e(G)$ and $\delta(G)$ denote the set of vertices, the set
of edges, the number of vertices, the number of edges and the minimum degree of $G$, respectively.
We also call $|G|$ and $e(G)$ the {\em order} and the {\em size} of $G$, respectively.
The
\emph{neighborhood set} of a vertex $v\in V(G)$ is $N_G(v)=\{u\in
V(G)\,|\,uv\in E(G)\}$, and its \emph{closed neighborhood set} is $N_G[v]=N_G(v)\cup \{v\}$.
%The \emph{degree} of a vertex $v$ in $G$ is
%denoted by $d(v)=|N_{G}(v)|$. Denote by $\delta(G)$ ($\Delta(G)$)
%the minimum degree (maximum degree) of the graph $G$.
For two subsets $X$ and $Y$ of
$V(G)$ we denote by $E_G[X,Y]$ the set of edges of $G$ with one endpoint
in $X$ and the other endpoint in $Y$.
%If $X=\{x\}$, we simply write
%$E_G[x,Y]$ for $E_G[\{x\},Y]$.
Let $N_G(v,S)=N_G(v)\cap S$.
Let
$\overline{G}$ be the complement of $G$. For any integers $m<n$, we
use the convenient notation $[m,n]$ for the set
$\{m,m+1,\cdots,n-1,n\}$, specially, we denote $[n]$ as $[1,n]$.
For a path $P_n=v_1v_2\ldots v_n$, we call $v_{\lceil\frac{n}{2}\rceil}$ and $v_{\lceil\frac{n+1}{2}\rceil}$ {\em center vertices} of $P_n$.
For a graph $H$ and a vertex $v\notin V(H)$, let $v\vee H$ be a graph obtained from $H$ by connecting each vertex of $H$ and $v$.
If $H$ is a path of order $n$, then $v\vee H$ is called a {\em kipas} with center $v$, and denote it by $\widehat{K}_n$.
\begin{definition}
For two graphs $G,H$, the \emph{Ramsey number} $r(G,H)$ is the minimum integer $n$ such that any red/blue edge-coloring of $K_n$ contains either a red copy of $G$ or a blue copy of $H$.
\end{definition}

\begin{definition}
Let $\mathcal{F}$ be a set of graphs. The \emph{Ramsey number} $r(G,\mathcal{F})$ is the minimum integer $n$ such that any red/blue edge-coloring of $K_n$ contains either a red copy of $G$ or a blue copy of graph in $\mathcal{F}$.
\end{definition}

For more details on Ramsey numbers, we refer to the survey paper \cite{Radzi} and some papers \cite{CH72, EMT71, FS, GG}.

\subsection{Ramsey and Gallai-Ramsey numbers}

The Gallai-$k$-coloring (or Gallai-coloring for short) is a $k$-edge-coloring of a complete graph without rainbow triangles.
The following result gives the structure of Gallai-coloring.
\begin{theorem}{\upshape \cite{Gallai,GS}}\label{K3}
In any rainbow triangle free edge-coloring of a complete graph, the vertex set can be partitioned into at least two parts such that between the parts there is a total of at most two colors, and between each pair of parts there is only one color on the edges.
\end{theorem}

%Gy\'{a}rf\'{a}s et al. \cite{GSS} obtained the following result.
%\begin{theorem}{\upshape \cite{GSS}}\label{K3bipartite}
%Let $H$ be a fixed graph with no isolated vertices. If $H$ is not
%bipartite, then $\operatorname{gr}_k(K_3; H)$ is exponential in $k$. If $H$ is bipartite, then $\operatorname{gr}_k(K_3; H)$ is
%linear in $k$.
%\end{theorem}

Let $g^q_k(p)$ be the smallest positive integer $n$ such that every Gallai-$k$-coloring of $K_n$ contains a $K_p$ receiving at most $q$ distinct colors.
In \cite{FGP}, Fox et al. studied the integer $p$ when $n$ is fixed, and got the following asymptotic result.
\begin{theorem}{\upshape \cite{FGP}}\label{Fox}
Let $k$ and $q$ be fixed positive integers with $q \leq k$. Every Gallai-$k$-coloring
of $K_n$ contains a set of order $\Omega(n^{\binom{q}{2}/\binom{k}{2}}\log_2^{c_{k,q}}n)$
which uses at most $q$ colors, where $c_{k,q}$ is
only depending on $k$ and $q$. Moreover, this bound is tight apart from the constant factor.
\end{theorem}
Note that $g^1_k(p)$ is the minimum integer $n$ such that every  Gallai-$k$-coloring of $K_n$ contains a monochromatic $K_p$.
The $g^1_k(p)$ is also denoted by $\operatorname{gr}_k(K_3: K_p)$.

\begin{definition}
For two graphs $G,H$, the \emph{Gallai-Ramsey number} $\operatorname{gr}_k(G:H)$ is defined as the minimum integer $N$ such that for any $n\geq N$ and any exact $k$-edge-coloring (using $k$ colors) of $K_n$ must contain either a rainbow copy of $G$ or a monochromatic copy of $H$.
\end{definition}

\begin{definition}
Define $\operatorname{gr}'_{k}(G: H)$ as the minimum integer $N$ such that for all $n\geq N$, every edge-coloring of $K_n$ using at most $k$ colors contains either a rainbow copy of $G$ or a monochromatic copy of $H$.
\end{definition}

It is obvious that $\operatorname{gr}'_{k}(G: H) \leq \max\{\operatorname{gr}_{i}(G: H): 1 \leq i \leq k\}$.
On the other hand, we have that $\operatorname{gr}'_{k}(G: H) \geq \max\{\operatorname{gr}_{i}(G: H): 1 \leq i \leq k\}$.
Otherwise, there exists an integer $k'$ ($1\leq k'\leq k$) such that $\operatorname{gr}_{k'}(G: H)>\operatorname{gr}'_{k}(G: H)$.
By the definition of $\operatorname{gr}_{k'}(G: H)$, there is a $k'$-edge-coloring of $K_{\operatorname{gr}_{k'}(G: H)-1} $ such that there is neither rainbow copy of $G$ nor monochromatic copy of $H$.
However, $\operatorname{gr}_{k'}(G: H)-1\geq \operatorname{gr}'_{k}(G: H)$, which implies that there is either a rainbow copy of $G$ or a monochromatic copy of $H$ in any $k'$-edge-coloring of $K_{\operatorname{gr}_{k'}(G: H)-1} $, a contradiction.
Therefore, $\operatorname{gr}'_{k}(G: H)= \max\{\operatorname{gr}_{i}(G: H): 1 \leq i \leq k\}$, and we only need to talk about the parameter $\operatorname{gr}_{k}(G: H)$.

Fox et al. \cite{FGP} proposed the following conjecture.
\begin{conjecture}{\upshape \cite{FGP}}
For integers $k\geq 3$ and $p\geq3$,
$$
\operatorname{gr}_k(K_3:K_p)=\left\{
\begin{array}{ll}
(r(K_p,K_p)-1)^{k/2}+1, &  \mbox{if }n\mbox{ is even},\\
(p-1)(r(K_p,K_p)-1)^{(k-1)/2}+1, & \mbox{if }n\mbox{ is odd}.\\
\end{array}
\right.$$
\end{conjecture}
Chung and Graham \cite{CG} and Gy\'{a}rf\'{a}s et al. \cite{GSS} proved the case $p=3$, Liu et al. \cite{LMSS} proved the case $p=4$, and Magnant and Schiermeyer \cite{MS-K5} proved the case $k=5$.
For more details on the Gallai-Ramsey numbers, we refer to a dynamic survey \cite{FMO} and the recent book \cite{MN-book}.

\subsection{Structural theorems and our methods}

The Gallai-$k$-coloring, a $k$-edge-coloring of $K_n$ without rainbow $K_3$, has an interesting structure.
Fujita and Magnant \cite{FM-K3+} consider the structures of $k$-edge-coloring of $K_n$ without rainbow $S_3^+$, where $S_3^+$ is a graph obtain from $K_3$ by adding a pendent edge.
Thomason and Wagner \cite{TW} got the structures of $k$-edge-coloring of $K_n$ without rainbow $P_5$.
In addition,
Gy\'{a}rf\'{a}s et al. \cite{GLST87} investigated the local Ramsey number of stars, which implies the structures of $k$-edge-coloring of $K_n$ without a rainbow star $K_{1,3}$ or a rainbow tree $P_4^+$, respectively, where $P_4^+$ is the
graph consisting of $P_4$ with one extra edge incident with a center vertex.

In this paper, we study the Gallai-Ramsey number $\operatorname{gr}_k(G:H)$, where $G\in \{K_{1,3},P_5,P_4^+\}$ and $H$ is a path or a kipas.
The following are the structures of $k$-edge-colored $K_n$ without rainbow $P_5,K_{1,3}$ and $P_4^+$, respectively.

Given an edge coloring of $K_n$, let $E_j$ be the set of edges colored $j$ and let $V_j$ be the vertices at which an edge of color $j$ appears.
The following are the structures of $k$-edge-coloring of $K_n$ without a rainbow $P_5$.
\begin{theorem}{\upshape \cite{TW}}\label{P5}
Let $k,n$ be two integers with $k\geq 4$ and $n\geq 5$. Let $K_n$ be a $k$-edge-colored graph so that it contains no rainbow $P_5$. Then, after renumbering the colors, one of the
following must hold:
\begin{itemize}
\item[] $(i)$ color 1 is dominant, meaning that the sets $V_j$, $j\geq2$, are disjoint;

\item[] $(ii)$ $K_n-a$ is monochromatic for some vertex $a$;

\item[] $(iii)$ there are three vertices $a,b,c$ such that $E_2=\{ab\}$,
$E_3=\{ac\}$, $E_4$ contains $bc$ plus perhaps some edges incident
with $a$, and every other edge is in $E_1$;

\item[] $(iv)$ there are four vertices $a,b,c,d$ such that $\{ab\}\subseteq
E_2\subseteq \{ab, cd\}$, $E_3=\{ac, bd\}$, $E_4=\{ad, bc\}$ and
every other edge is in $E_1$;

\item[] $(v)$ $n=5$, $V(K_n) = \{a, b, c, d, e\}$, $E_1 =\{ad, ae, bc\}$,
$E_2=\{bd, be, ac\}$, $E_3=\{cd, ce, ab\}$ and $E_4=\{de\}$.
\end{itemize}
\end{theorem}

Let $G_1(n)$ be a $3$-edge-coloring of $K_n$ such that $V(K_n)$ can be partitioned into three parts $V_1,V_2,V_3$ (at most one of them is empty) such that each edge of $G[V_1]$ is colored by either $1$ or $3$, each edge of $G[V_2]$ is colored by either $1$ or $2$, and each edge of $G[V_3]$ is colored by either $2$ or $3$.
Moreover, each edge between $V_1,V_2$ is colored by $1$, each edge between $V_2,V_3$ is colored by $2$ and each edge between $V_1,V_3$ is colored by $3$.

Let $G_2(n)$ be a $4$-edge-coloring of $K_n$ such that there is exactly one edge, say $xy$, having
color $2$. All other edges incident with $x$ are colored by $3$, and all other edges incident with $y$ are colored by $4$. All edges not incident to vertices $x, y$ have color $1$.

Let $G_3(n)$ be a $4$-edge-coloring of $K_n$ such that there exists a rainbow $K_3$ (say the vertices in the triangle are $a,b,c$, where $ab$ has color $2$, $bc$ has color $3$ and $ac$ has color
$4$). In addition, any other edge has color $1$.

\begin{remark}\label{remark-1}
If $\Gamma$ is an edge-coloring of $\{(ii),(iii),(iv),(v),G_2(n),G_3(n)\}$ of $K_n$, then the graph induced by edges with color $i\geq 2$ is a star forest.
\end{remark}

The following are the structures of $k$-edge-coloring of $K_n$ without a rainbow $K_{1,3}$.
\begin{theorem}[\cite{GLST87}]\label{K13}
For positive integers $k\geq 3$ and $n$, if $G$ is a $k$-edge-coloring of
$K_n$ without rainbow $K_{1,3}$, then after renumbering the colors,
one of the following holds.

$(I1)$ $k=3$ and $G\cong G_1(n)$;

$(I2)$ $k\geq 4$ and Item $(i)$ in Theorem
\ref{P5} holds.
\end{theorem}

The following are the structures of $k$-edge-coloring of $K_n$ without rainbow $P_4^+$.

\begin{theorem}[\cite{GLST87}]\label{P4+}
For positive integers $k\geq 4$ and $n$, if $G$ is a $k$-edge-coloring of
$K_n$ without rainbow $P_{4}^{+}$, then after renumbering the colors,
one of the following holds.

$(J1)$ $k=4$ and $G\in \{G_2(n),G_3(n)\}$;

$(J2)$ $k\geq 4$ and Item $(i)$ in Theorem
\ref{P5} holds.
\end{theorem}

Recall Theorem \ref{K3} and definition of $\operatorname{gr}_k(K_3:H)$, the structure of Gallai-coloring is important for $\operatorname{gr}_k(K_3:H)$.
Analogously, the structures of $k$-edge-colorings in Theorem \ref{P5} is important for $\operatorname{gr}_k(P_5:H)$.
However, there are five structures, and it is complicated if we consider the five structures at the same time.
Observe that the structures $(ii)-(iv)$ are simple since they contain a large monochromatic complete graph (the monochromatic complete graph can obtained by deleting at most three vertices), and structure $(v)$ is simple  since there are only five vertices.
Therefore, the structures $(i)$ is very important for $\operatorname{gr}_k(P_5:H)$.
Similarly, in Theorem \ref{K13}, the structures $G_1(n)$ and $(i)$ are important;
in Theorem \ref{P4+}, the structure $(i)$ is important and structures $G_2(n)$ and $G_3(n)$ are clear.
By above discussion, $G_1(n)$ and $(i)$ are key structures in $k$-edge-colored $K_n$ without rainbow $P_5,K_{1,3}$ or $P_4^+$.

Firstly, we observe the structure $(i)$.
Suppose $k\geq 3$.
It is clear that $K_n$ can be partitioned into $k-1$ parts $V_1,\cdots,V_{k-1}$ such that each edge between different parts is colored by $1$, and each edge of $G[V_{i-1}]$ is colored by either $1$ or $i$.
Therefore, $G[V_{i}]$ contains all edges of color $i+1$ and $|V_i|\geq 2$ for $i\in[k-1]$.
We denote the set of all such $k$-edge-colorings on $K_n$ by $\mathcal{B}_k(n)$.
\begin{definition}
Let $b_k(H)$ be the minimum number $n$ such that any edge-coloring $\Gamma\in \mathcal{B}_k(n)$ of $K_n$ contains a monochromatic $H$.
\end{definition}

For the structure $G_1(n)$, $V(K_n)$ can be partitioned into three parts $V_1,V_2,V_3$ as the definition, and at most one part is empty.
If there is an empty part, then $G_1(n)\in \mathcal{B}_3(n)$.
Thus, we only consider that $V_1,V_2,V_3$ are all nonempty sets.
We denote the set of all such 3-edge-colorings on $K_n$ by $\mathcal{T}(n)$.
\begin{definition}
Let $t(H)$ be the minimum number $n$ such that any edge-coloring $\Gamma\in \mathcal{T}(n)$ of $K_n$ contains a monochromatic $H$.
\end{definition}

\begin{remark}
In order to determine the $\operatorname{gr}_k(G:H)$, where $G\in \{P_5,K_{1,3},P_4^+\}$, the main work is to determine $b_k(H)$ or $t(H)$.
It is easy to verify whether $K_n$ contains a monochromatic copy of $H$ under the other structures.
\end{remark}

\subsection{Progress and our results}

Li et al. \cite{LWL} got some exact values of $\operatorname{gr}_k(P_5 : K_t)$ for small $t$ and bounds for general $t$. Recently, Wei et al. \cite{WHMZ22} got the exact values of $\operatorname{gr}_k(P_5 : H) \ (k \geq 3)$ for small $t$ and upper and lower bounds for general $t$, where $H$ is either a general fan or a general wheel graph.
The motivation of this paper is whether we can obtain exact values for $\operatorname{gr}_k(P_5 : H) \ (k \geq 4)$ such that $H$ is not just a graph but belongs to a class of graphs.

In this paper, we first study the new parameters $b_k(P_n), t(P_n),b_k(\widehat{K}_n)$ and $t(\widehat{K}_n)$, and further get the exact values of $\operatorname{gr}_k(G:P_n)$ and $\operatorname{gr}_k(G:\widehat{K}_n)$, where $G\in \{P_5,K_{1,3},P_4^+\}$.

We call a linear forest a {\em $k$-linear forest} if any component of the linear forest is a path of order at least $k$.
We first get a classical Ramsey number for $\widehat{K}_n$ and the set of $3$-linear forests, and the result is useful and will be used in Section \ref{sec-kipas}.
\begin{theorem}\label{class}
If $2\leq m\leq \frac{n}{2}$ and $\mathcal{L}$ is the set of linear forest $L$ of size at least $m$, 
then
$$
r(\widehat{K}_n,\mathcal{L})= n+ \lceil m/2\rceil.
$$
In addition, if $6\leq m\leq \frac{n}{2}$ and $\mathcal{L}'$ is the set of $3$-linear forest $L$ of size at least $m$, 
then
$$
r(\widehat{K}_n,\mathcal{L}')= n+ \lceil m/2\rceil.
$$
\end{theorem}
The Gallai-Ramsey numbers are list as follows.
\begin{theorem} \label{main-p5-paths}
If $n\geq k\geq 4$, then
$$
\operatorname{gr}_k(P_5:P_n)=\left\{
\begin{array}{ll}
n+1, &{\rm if}~2(k-1)\leq n\leq 4(k-2)+1,\\
\left\lceil\frac{3n-3}{2}\right\rceil, &{\rm if}~n>4(k-2)+1.\\
\end{array}
\right.$$
\end{theorem}

\begin{theorem} \label{main-p4+-paths}
If $n\geq k\geq 4$, then
$$
\operatorname{gr}_k(P_4^+:P_n)=\left\{
\begin{array}{ll}
n+2, &{\rm if}~k=4~{\rm and }~2(k-1)\leq n\leq 4(k-2)+1,\\
n,   & {\rm if}~k\geq 5~{\rm and }~ 2(k-1)\leq n\leq 4(k-2)+1,\\
\left\lceil\frac{3n-3}{2}\right\rceil, &{\rm otherwise}.\\
\end{array}
\right.$$
\end{theorem}

\begin{theorem}\label{main-k13-paths}
$(1)$ If $n\geq 2(k-1)$ and $k\geq 4$, then
$$
\operatorname{gr}_k(K_{1,3}:P_n)=\left\{
\begin{array}{ll}
n, &{\rm if}~2(k-1)\leq n\leq 4(k-2)+1,\\
\left\lceil\frac{3n-3}{2}\right\rceil, &{\rm if}~n>4(k-2)+1.\\
\end{array}
\right.
$$

$(2)$ If $k=3$ and $n\geq 4$, then
$$
\operatorname{gr}_k(K_{1,3}:P_n)=\left\{
\begin{array}{ll}
\frac{3n}{2}-1, &{\rm if}~n\mbox{ is even},\\
\frac{3n-1}{2}, &{\rm if}~n\mbox{ is odd }.
\end{array}
\right.
$$
\end{theorem}

\begin{remark}
We do not discuss $\operatorname{gr}_k(P_5:P_n),\operatorname{gr}_k(P_4^+:P_n)$
and $\operatorname{gr}_k(K_{1,3}:P_n)$ when $k\geq 4$ and $n<2(k-1)$, since it is easily to obtain.
In fact, if $k\geq 4$ and $n<2(k-1)$, then $K_n$ must contain a monochromatic $P_n$ when the edge-coloring $\Gamma$ of $K_n$ belongs to $\mathcal{B}_k(n)$.
However, other edge-colorings $G_2(N),G_3(N),(ii),(iii),(iv)$ and $(v)$ are simple.
\end{remark}

\begin{theorem}\label{main-kipas}
For $n\geq 5$, we have that $
\operatorname{gr}_3(K_{1,3}:\widehat{K}_n)=\left\lfloor\frac{5n}{2}\right\rfloor$ if $n$ is odd, and $\left\lfloor\frac{5n}{2}\right\rfloor-1\leq \operatorname{gr}_3(K_{1,3}:\widehat{K}_n)\leq \left\lfloor\frac{5n}{2}\right\rfloor$ if $n$ is even.
\end{theorem}

\section{Preliminaries}

As we know, the study of Gallai-Ramsey numbers deeply depend on the classical Ramsey number.
Therefore, we list some useful classical Ramsey numbers in this section, which can be found in the survey paper \cite{Radzi}.

\begin{theorem}{\upshape \cite{GG}}\label{path-path}
If $2\leq m\leq n$, then
$r(P_n,P_m)= n+\left\lfloor m/2\right\rfloor -1$.
\end{theorem}

The following theorem is a generalization of Theorem \ref{path-path}.
\begin{theorem}{\upshape\cite{FS}}\label{r-linear}
Suppose $L_1,L_2$ are $2$-linear forests having $j_1,j_2$ components of odd order, respectively.
Then $$r(L_1,L_2)=\max\left\{|L_1|+\left\lfloor\frac{|L_2|-j_2}{2}\right\rfloor-1,
|L_2|+\left\lfloor\frac{|L_1|-j_1}{2}\right\rfloor-1\right\}.$$
\end{theorem}

\begin{theorem}{\upshape \cite{Ptd}}\label{star-path}
$r(P_m,K_{1,n})= m+n-1$ if $n\equiv 1 \pmod {n-1}$, and $r(P_m,K_{1,n})\leq m+n-2$ otherwise.
\end{theorem}

\begin{theorem}{\upshape \cite{Hara}}\label{star-star}
If $m,n\geq 2$, then
$$
r(K_{1,n},K_{1,m})=\left\{
\begin{array}{ll}
m+n-1, &  \mbox{ if }m,n \mbox{ are even},\\
m+n, & \mbox{ otherwise}.\\
\end{array}
\right.
$$
\end{theorem}

Li et al. \cite{LZH} got a closing gap on path-kipas Ramsey numbers.
\begin{theorem}{\upshape \cite{LZH}}\label{path-kipath-1}
For $m\leq 2n-1$ and $m,n\geq 2$, we have
$$
r(P_n,\widehat{K}_m)=\max\left\{2n-1,\left\lceil\frac{3m}{2}\right\rceil-1,2\left\lfloor\frac{m}{2}\right\rfloor+n-2\right\}.
$$
\end{theorem}

Salman and Broersma \cite{SB} obtained the following result for path-kipas Ramsey numbers.
\begin{theorem}{\upshape \cite{SB}}\label{path-kipath-2}
If $n\geq 4$ and $m\geq 2n-2$, then
$r(P_n,\widehat{K}_m)\leq m+n-1$.
\end{theorem}

Li et al. \cite{LZB} obtained an exact formula for all star-kipas Ramsey numbers.
\begin{theorem}{\upshape \cite{LZB}}\label{star-kipath}
Let $n,m\geq 2$ be two integers.

$(1)$ If $m\geq 2n$, then
$$
r(K_{1,n},\widehat{K}_m)=\left\{
\begin{array}{ll}
m+n-1, &  \mbox{ if }m,n \mbox{ are even},\\
m+n, & \mbox{ otherwise}.\\
\end{array}
\right.
$$

$(2)$ If $m\leq 2n-1$, then
$$
r(K_{1,n},\widehat{K}_m)=\left\{
\begin{array}{ll}
2n+\lfloor\frac{m}{2}\rfloor-1, &  \mbox{ if } m,\lfloor\frac{m}{2}\rfloor \mbox{ are even},\\
2n+\lfloor\frac{m}{2}\rfloor, & \mbox{ otherwise}.\\
\end{array}
\right.
$$
\end{theorem}

The following is a result about decomposition of the complete graph.
\begin{theorem}{\upshape \cite{L}}\label{signi}
For $n\geq1$, $K_{2n+1}$ can be decomposed into $n$ edge-disjoint Hamiltonian cycles; $K_{2n+2}$ can be decomposed into $n$ edge-disjoint Hamiltonian cycles and a perfect matching.
\end{theorem}

The following result will be used in the subsequent proofs, and we give a simple proof here, although it may be proved somewhere.
\begin{lemma}\label{mutil-HC}
Suppose that $G$ is a complete $r$-partite graph with parts $V_1,\ldots,V_r$ and $V_r$ is one maximum part.
\begin{enumerate}
  \item If $\sum_{i\in[r-1]}|V_i|\geq |V_r|$, then $G$ contains a Hamiltonian cycle.
  \item If $\sum_{i\in[r-1]}|V_i|=|V_r|-1$, $G$ contains a Hamiltonian path.
\end{enumerate}
\end{lemma}
\begin{proof}
The result is true for bipartite graphs.
So, suppose $r\geq 3$.
Without loss of generality, let $|V_1|\leq |V_2|\leq\ldots\leq |V_r|$. Then there exists an $s\in[r]$ such that $|V_{s-1}|<|V_s|=\ldots=|V_t|$.
The proof proceeds by induction on $|G|$.
If $|G|=3$, then the result holds obviously.
If $|V_r|=1$, then $G$ is a complete graph and the result follows.
So, suppose that  $|G|\geq 4$ and $|V_r|>1$.
Choose a vertex $v_i$ from $V_i$ for $i\in[s,r]$.
Let $U_i=V_i$ for $i<s$ and $U_i=V_i-v_i$ for $i\in[s,r]$, and let $G'=G-\{v_s,\ldots,v_r\}$.
Then $U_r$ is a maximum part in $G'$.
Since $|V_i|\geq 2$, each $U_i$ is nonempty.

We prove the first statement. Since $\sum_{i\in[r-1]}|V_i|\geq |V_r|$, it follows that $\sum_{i\in[r-1]}|U_i|\geq |U_r|$. By induction, $G'$ contains a Hamiltonian cycle $C$.
Note that $v_s,\ldots,v_r$ induces a complete graph.
Since $r\geq 3$, there are three integer $i,j,k$ such that there are two adjacent edges $xy$ and $yz$ of $G$ with $x\in V_i$, $y\in V_j$ and $z\in V_k$.
If $r=s$, then we can assume that $r\notin\{i,j\}$, and hence $(C\backslash\{xy\})\cup\{xv_r,yv_r\}$ is a Hamiltonian cycle of $G$.
If $s<r-1$ (resp. $s=r-1$), then there is a cycle $C'=v_sv_{s+1}\ldots v_rv_s$ in $G$ (resp. an edge $v_{r-1}v_r$ of $G$).
There is an integer of $i,j,k$ which is not in $\{r-1,r\}$ (say $i\notin\{r-1,r\}$), and one of $r-1,r$ is not equal to $j$ (without loss of generality, suppose that $j\neq r$). 
Then $(C\backslash\{xy\})\cup (C'\backslash\{v_rv_{r-1}\})\cup\{xv_{r-1},yv_r\}$ (resp. $(C\backslash\{xy\})\cup\{v_{r-1}v_r,xv_{r-1},yv_r\}$) is a Hamiltonian of $G$.

Now we prove the second statement.
In this case, $s=r$ and $\sum_{i\in[r-1]}|U_i|=|U_r|$.
So, $G'$ has a Hamiltonian cycle $C^*$ by the first result.
Let $ww'\in E(C^*)$ and suppose that $w\notin V_r$ by symmetry.
Then $(C^*\backslash\{ww'\})\cup\{wv_r\}$ is a Hamiltonian path of $G$.
\end{proof}

\section{Technique results and the proof of Theorem \ref{class}}

In this section, we prove two critical lemmas, which will be used in Section \ref{sec-kipas}.
\begin{lemma}\label{clm-0}
Suppose that $n\geq 4$, $a\in\{1,2\}$ and $N=n+a$.
If $\Gamma$ is a red/blue edge-coloring of $K_N$ such that $K_N$ does not contain a red copy of $\widehat{K}_n$, then the following results hold.

$(i)$ If $a=1$, then $K_N$ contains either a blue copy of $2P_2$ or a blue copy of $P_3$.

$(ii)$ If $a=2$ and $n\geq 5$, then $K_N$ contains either a blue copy of $2P_3$, or a blue copy of $P_5$, or a blue copy of vertex-disjoint $P_4$ and $P_2$.
\end{lemma}
\begin{proof}
Suppose that $R,B$ are the graphs with the common vertex set $V=V(K_N)$ and their edges sets are red edges and blue edges, respectively.
Note that $R$ does not contains a copy of $\widehat{K}_n$.
For convenience, we use $P_2\cup P_4$ to denote vertex-disjoint $P_4$ and $P_2$ throughout this proof.

$(i)$ For $a=1$, if $\delta(B)\geq 1$, then $B$ contains either a $2P_2$ or a $P_3$, since $n\geq 4$. If $\delta(B)=0$, then there exists a vertex in $B$ such that $d_{B}(v)=0$, and hence $R-v$ does not contain a Hamiltonian path (otherwise there is a red $\widehat{K}_n$, a contradiction).
Suppose that $P$ is the maximum path of $R-v$ with endpoints $a,b$.
Then $V(R)-\{v\}-V(P)\neq \emptyset$.
By the maximality of $P$, we have $ua,ub\in E(B)$ for each $u\in V(R)-\{v\}-V(P)$, and hence $B$ contains a $P_3$.

$(ii)$ 
For $a=2$, suppose to the contrary that $B$ does not contain $2P_3$, $P_5$ or $P_4\cup P_2$.
We first assume that $B$ has at most one vertex of degree greater than one (if the vertex exists, denote it by $w$). It is clear that $|B-w|=n+1$ and $\Delta(B-w)\leq 1$ if $w$ exists. If $w$ exists and $B-w$ contains a perfect matching $M$, then there is a blue $P_5$ with center $w$. 
Otherwise, there are two vertex $u,v$ of $K_N$ such that $u$ is an isolated vertex of $B-v$ and $\Delta(B-v)\leq 1$.
In fact, if $w$ exists and $B-w$ does not contain a perfect matching $M$, then there is an isolated vertex $u$ of $B-w$ (we can see $w$ as $v$) and $\Delta(B-v)\leq 1$;
if $w$ does not exist, then $\Delta(B-w)\leq 1$, and for an edge $xy$ of $B$, $x$ is an isolated vertex of $B-y$ (we can see $x,y$ as $u,v$, respectively).
Hence, the vertices $u,v$ are obtained. Let $V'=V-\{u,v\}$ and $R'=R-\{u,v\}$. Then $u\vee V'$ is a red star. Since $\Delta(B-\{u,v\})\leq 1$, it follows that $\delta(R')\geq|V'|-2$.
Since $|V'|=n\geq 5$, by Dirac's Theorem, $R'$ contains a Hamiltonian path $P$.
Hence, $u\vee P$ is a red $\widehat{K}_n$, a contradiction. 

Now we assume that there are at least two vertices, say $z_1,z_2$, of degree at least two in $B$. Without loss of generality, we can choose $z_1,z_2$ such that
$d_{B}(z_1)+d_{B}(z_2)$ is as large as possible.
We use $E_u$ to denote the graph induced by blue edges incident with $u$.
If $N_B(z_1)\cap N_B(z_2)=\emptyset$, then $E_{z_1}\cup E_{z_2}$ contains a $2P_3$, unless $E_{z_1}\cup E_{z_2}$ is a double star.
If $|N_B(z_1)\cap N_B(z_2)|=1$, then $E_{z_1}\cup E_{z_2}$ contains a $P_5$, unless $E_{z_1}\cup E_{z_2}$ is a graph obtained form $K_{1,k}$, $k\geq 2$, by adding an edge of $\overline{K_{1,k}}$ (the resulting graph is denoted by $K_{1,k}^+$). 
If $|N_B(z_1)\cap N_B(z_2)|=2$, then $E_{z_1}\cup E_{z_2}$ contains a $P_5$, unless $E_{z_1}\cup E_{z_2}$ is either a $K_4$, or a $K_4^-$, or a $C_4$, where $K_4^-$ is a graph obtained from $K_4$ by deleting an edge.
If $|N_B(z_1)\cap N_B(z_2)|=3$, then $E_{z_1}\cup E_{z_2}$ contains a $P_5$.
Recall that $B$ does not contain $2P_3,P_5$ or $P_2\cup P_4$ by assumption.
Hence, either $E_{z_1}\cup E_{z_2}$ is a double star or $E_{z_1}\cup E_{z_2}\in \{K_{1,k}^+,K_4,K_4^-,C_4\}$.
For convenience, let $\widetilde{H}$ denote the subgraph of $B$ induced by $E_{z_1}\cup E_{z_2}$.
Then either $\widetilde{H}$ is a double star or $\widetilde{H}\in \{K_{1,k}^+,K_4,K_4^-,C_4\}$.

If $\widetilde{H}\in \{K_4,K_4^-,C_4\}$, then each edge between $U=V(\widetilde{H})$ and $V-U$ is red; otherwise there is a blue $P_5$, a contradiction.
It is obvious that $G-U$ is a red complete graph; otherwise there is a blue $P_2\cup P_4$, a contradiction.
Choose a vertex $u\in V-U$, $R-\{u,z_1\}$ contains a red $P_n$ when $n=5$. 
For $n\geq 6$, since $\Delta(B-\{u,z_1\})=2$, it follows that $\delta(R-\{u,z_1\})\geq n-3\geq \frac{n}{2}$.
By Dirac's Theorem, $R-\{u,z_1\}$ contains a red $P_n$.
Note that $u$ connects every other vertex by red an edge.
For $n\geq 5$, there is a red $\widehat{K}_n$, a contradiction.

If $\widetilde{H}=K_{1,2}^+$ (in other words, $\widetilde{H}$ is a triangle), then $d_B(z_1)=d_B(z_2)=2$.
without loss of generality, suppose that $N_B(z_1)\cap N_B(z_2)=\{z\}$.
By the maximality of $d_B(z_1)+d_B(z_2)$, we have that $d_B(z)=2$.
Since $B$ does not contain $2P_3$, $\Delta(B-\{z_1,z_2,z\})\leq 1$.
If $n=5$, then there is a vertex $u$ of $V-\{z_1,z_2,z\}$ such that $d_B(u)=0$.
Since all edges between $\widetilde{H}$ and $V-u-\widetilde{H}$ are red, it is clear that there is a red $P_n$ between $\widetilde{H}$ and $V-u-\widetilde{H}$.
Hence, $u\vee P_n$ is a red $\widehat{K}_n$, a contradiction.
If $n\geq 6$, then choose a vertex $u\in V-\{z_1,z_2,z\}$ and let $U=V-N_B(u)$. Since $d_B(u)\leq 1$, it follows that $|U|\geq n$, $\Delta(B[U])\leq 1$ and $u$ connects every vertex of $U$ by an red edge.
Then $\delta(R[U])\geq n-3\geq \frac{n}{2}$.
Hence, $R[U]$ contains a Hamiltonian path $P_n$, and $R$ contains a $\widehat{K}_n$, a contradiction.

If $\widetilde{H}=K_{1,k}^+$, where $k\geq 3$, then we can assume that the edges in $\widetilde{H}$ are 
$$z_1x_1,\ldots,z_1x_{k-1},z_1z_2,z_2x_{k-1}.$$
It is obvious that $\Delta(B-z_1)$ contains only one edge $z_2x_{k-1}$; otherwise there is either a blue $P_5$ or a blue $P_2\cup P_4$, a contradiction. 
Therefore, $\delta(R-\{z_1,x_1\})=n-2\geq \frac{n}{2}$. 
By Dirac's Theorem, $R-\{z_1,x_1\}$ contains a Hamiltonian path $P_n$.
Since $x_1$ is an isolated vertex in $B$, it follows that $R$ contains a $\widehat{K}_n$, a contradiction.

If $\widetilde{H}$ is a double star, then $N_B(z_1)\cap N_B(z_2)=\emptyset$.
If $B-\{z_1,z_2\}$ contains an edge $e$, then $E_{z_1}\cup E_{z_2}$ contains either a $P_5$ or a $P_2\cup P_4$, unless $E_{z_1}\cup E_{z_2}$ is a $P_4$  and $e=xy$ (say the edges in $E_{z_1}\cup E_{z_2}$ are $xz_1,z_1z_2$ and $z_2y$).
By the maximality of $d_B(z_1)=d_B(z_2)$, we have that $d_B(x)=d_B(y)=d_B(z_1)=d_B(z_2)=2$, and hence $B$ is a $C_4$ with $E(B)=\{z_1z_2,xz_1,z_2y,xy\}$. Recall that the case have discussed since $B=E_{z_1}\cup E_y$ is a $C_4$. 
Therefore, assume that $B-\{z_1,z_2\}$ is the empty graph. Since  $|N_B(z_1)-\{z_2\}|\cup |N_B(z_2)-\{z_1\}|\leq N+2=n$, either $|N_B(z_1)-\{z_2\}|\leq \frac{n}{2}$ or $|N_B(z_2)-\{z_1\}|\leq \frac{n}{2}$ (say $|N_B(z_2)-\{z_1\}|\leq \frac{n}{2}$).
Then $|V-z_1-N_B[z_2]|\geq n+2-(\frac{n}{2}+2)=\frac{n}{2}\geq 2$.
Choose two vertices $x,y$ of $V-z_1-N_B[z_2]$.
It is obvious that $x$ connects every vertex of $V-z_1$ by a red edge and $yz_2$ is a red edge.
Note that $R[V-\{z_1,z_2,x\}]$ is a complete graph.
Hence, there is a Hamiltonian path $Q=P_{n-1}$ of $R[V-\{z_1,z_2,x\}]$ with one endpoint $y$, and hence $Q\cup \{yz_2\}$ is a red $P_n$.
Now we obtain a red $\widehat{K}_n=x\vee (Q\cup \{yz_2\})$, a contradiction.
\end{proof}

Let $\mu_L$ be the number components of $3$-linear forest $L$. Note that $e(L)=|L|-\mu_L$.
\begin{lemma}\label{lem-1}
Suppose that $N=n+a$ and $3\leq a\leq \lfloor n/4\rfloor$.
For any red/blue edge-coloring of $K_N$, there is either a red copy of $\widehat{K}_n$ or a blue copy of $3$-linear forest $L$ such that $|L|-\mu_L\geq 2a$.
\end{lemma}
\begin{proof}
%The proof is proceeded by contradiction.
Suppose that $G$ is a red/blue edge-colored $K_{N}$ containing no red $\widehat{K}_n$.
Our aim is to obtain a blue $3$-linear forest $L$ with $|L|-\mu_L\geq 2a$ in $G$. Let $G_1$ and $G_2$ be graphs induced by red edges and blue edges, respectively. If $G_2$ does not contain $3$-linear forest, then $\Delta(G_2)\leq 1$. We can choose a vertex $u$ of $G_1$ and $S\subseteq V(G)-N_G[u]$ such that $|S|=n$.
Since $\delta(G_1[S])\geq n-2>n/2$, $G_1[S]$ contains a $P_n$. Hence $G$ has a $\widehat{K}_n$ with center $u$, a contradiction.
Therefore, $E(L)\neq \emptyset$.
We choose a $3$-linear forest $L$ from $G_2$ such that $e(L)=|L|-\mu_L$ is maximum.
Subjects to above, suppose that $|L|$ is maximum.
%Then $|L|-\mu_L\leq 2a-1$.

Let $U=V(G)-V(L)$.
Then the following result holds.
\begin{claim}\label{clm-1}
For each component $D$ of $L$ (say, $D=P_m$), the following results hold.
\begin{itemize}
\item[] $(i)$ If $m\geq 6$, then $E_G[V(D),U]\subseteq E(G_1)$.

\item[] $(ii)$ If $3\leq m\leq 5$, then either $E_G[U,V(D)]\subseteq E(G_1)$ or there is a vertex $v$ of $D$ such that $E_G[U,V(D)-v]\subseteq E(G_1)$.
    Moveover, if $v$ exists, then $v$ is a center vertex of $D$ and $G_1[V(D)-\{v\}]$ contains a Hamiltonian path.
\end{itemize}
\end{claim}
\begin{proof}
Suppose that $D=v_1v_2\ldots v_m$. We first prove the following result.
\begin{fact}\label{new-fact}
For each $u\in U$ and $i\in[m-1]$, $v_1u,v_mu\not \in E(G_2)$ and at most one of $uv_i,uv_{i+1}$ belongs to $G_2$.
\end{fact}
\begin{proof}
If $v_1u\in E(G_2)$ or $v_mu\in E(G_2)$, then we can find a $3$-linear forest $L\cup\{v_1u\}$ or $L\cup\{v_mu\}$ which is larger than $L$, a contradiction.
For $i\in [m-1]$, if $uv_i,uv_{i+1}\in E(G_2)$, then we can get a larger $3$-linear forest $(L\setminus  v_iv_{i+1})\cup \{uv_i,uv_{i+1}\}$ in $G_2$, a contradiction.
\end{proof}

$(i)$ Suppose to the contrary that there exists some $i \in[2,m-1]$ and $u\in U$ such that $v_iu\in E(G_2)$. Without loss of generality, suppose that $2\leq i\leq m/2$.
Then $(D\backslash \{v_iv_{i+1}\})\cup\{v_iu\}$ is a new $3$-linear forests of $G_2$ with $e(L')=e(L)$ and $|L'|=|L|+1$, which contradicts the choose of $L$.

$(ii)$
If $E_G[U,V(D)]\subseteq E(G_1)$, then the result holds. Thus, suppose that there exists a vertex $v_i\in V(D)$ and a vertex $u'\in U$ such that $v_iu'\in E(G_2)$.

If $m=3$, then by Fact \ref{new-fact}, $i=2$ and $v_2$ is the unique vertex such that $E_G[V(D)-v_2,U]\subseteq E(G_1)$.
Moreover, we have $v_1v_3\not\in E(G_2)$; otherwise $u'v_2v_1v_3$ is a path of $G_2$ and we can find a $3$-linear forest $(L-D)\cup u'v_2v_1v_3$ larger than $L$, a contradiction.
Thus, $v_1v_2$ is a Hamiltonian path of $G_1[V(D-v_2)]$.

If $m=4$,  then by Fact \ref{new-fact}, $i\in\{2,3\}$. 
Without loss of generality, assume that $i=2$.
If there is a vertex $z$ of $U$ with $v_3z\in E(G_2)$, then $z\neq u'$ by Fact \ref{new-fact}.
Hence, $E_G[V(D)-v_2,U]\subseteq E(G_1)$.
By the same discussion as above, we get that $v_1v_4,v_1v_3\in E(G_1)$; otherwise there is a $3$-linear forest larger than $L$, a contradiction. So, $v_3v_1v_4$ is a Hamiltonian path of $G_1[V(D-v_2)]$.

If $m=5$,  then $i\notin\{2,4\}$.
Otherwise, either $L'=(L\backslash \{v_2v_3\})\cup\{v_2u'\}$ or $L'=(L\backslash \{v_4v_3\})\cup\{v_4u'\}$ is a
$3$-linear forest with $e(L)=e(L')$ and $|L|<|L'|$, a contradiction.
Hence, $i=3$ and $E_G[V(D)-v_3,U]\subseteq E(G_1)$.
It is obvious that $v_1v_5,v_2v_5,v_1v_4\in E(G_1)$; otherwise, we can find a larger $3$-linear forest than $L$, a contradiction.
Therefore, $v_4v_1v_5v_2$ is a Hamiltonian path of $G_1[V(D-v_3)]$.
\end{proof}

We call the unique vertex $v$ in Claim \ref{clm-1} the {\em removable vertex} of $D$.
Note that our aim is to prove that $|L|-\mu_L\geq 2a$. 
For convenience, the proof is proceeded by contradiction, that is, $|L|-\mu_L\leq 2a-1$.
In the following proof,
in order to get a contradiction, we need to find a red $\widehat{K}_n$.

If $U=\emptyset$, then $|L|=N$ and $|L|-\mu_L\geq \frac{2|L|}{3}=\frac{2N}{3}$.
Since $a\leq \frac{n}{4}$, it follows that $|L|-\mu_L\geq 2a$, the result follows.
Thus, suppose that $U\neq \emptyset$.
Choose a vertex $u_0\in U$ such that $|N_{G_2}(u_0,U)|$ is as small as possible.
By the maximality of $L$, we have that $\Delta(G_2[U])\leq 1$, and hence $|N_{G_2}(u_0,U)|\leq 1$.
Moreover, if $|N_{G_2}(u_0,U)|=1$, then $G_2[U]$ is a perfect matching of $G[U]$.
Let $U'=U-u_0-N_{G_2}(u_0,U)$.
Then $|U'|\geq |U|-2$.

In order to get a contradiction, we need to find a red $\widehat{K}_n$ with center $u_0$, that is, to find a path $P_n$ in $G_1[N_{G_1}(u_0)]$.

Let $\mathcal{D}_1$ be the set of components $D$ of $L$ such that $E_G[V(D),U]\subseteq E(G_1)$, and let $\mathcal{D}_2$ be the set of other components of $L$.
By Claim \ref{clm-1}, for each $D\in \mathcal{D}_2$, we have $3\leq |D|\leq 5$.

The following  claim is immediate, which will be used later.

%\begin{claim}\label{cr}
%$\mathcal{D}_2\neq\emptyset$.
%\end{claim}
%\begin{proof}
%Assume, to the contrary, that $\mathcal{D}_2=\emptyset$. Then $E_G[V(L),U]\in E(G_1)$.
%Since $L$ is a $3$-linear forest and $|L|-\mu_L\leq 2a-1$, it follows that $|L|-\mu_L\geq \frac{2|L|}{3}$. Since $|L|-\mu_L\leq 2a-1$ and $a$ is an integer, $|L|\leq 3a-2$ and $|U'|\geq N-2-|L|\geq n-2a$.
%Since $a\leq n/4$, $|U'|\geq n/2$.
%If $|L|\geq |U'|$, then there is a red $P_n$ between $U'$ and $V(L)$, and hence $G_1$ contains a $\widehat{K}_n=u_0\vee P_n$, a contradiction.
%If $|L|< |U'|$, let $U''$ be a subset of $U'$ with $|U''|=|L|$.
%
%
%
%For each component $D\in \mathcal{D}_1$, we choose a vertex $v_D\in V(D)$.
%Let $U''=\{v_D:D\in \mathcal{D}_1\}$.
%Then $|U''|=\mu_L$ and $G_1[U''\cup U']$ is a complete graph.
%Note that 
%\begin{align*}
%|U''\cup U'|&\geq (N-|L|-2)+\mu_L=N-(L-\mu_L)-2\\
%&\geq N-(2a-1)-2=n-a-1\\
%&\geq 2a-1\geq |L|.
%\end{align*}
%Hence, $G_1[U''\cup U']$ contains a red path of order larger than $N-2\geq n$, and so $G_1$ contains a red $\widehat{K}_n$ with center $u_0$, a contradiction.
%\end{proof}

\begin{claim}\label{clm-2}
If $|N_{G_2}(u_0,U)|=1$, then there is no $P_3$ in $\mathcal{D}_2$.
\end{claim}
\begin{proof}
Assume, to the contrary, that $\mathcal{D}_2$ contains a path $P=v'vv''$. Then there exists a vertex $u\in U$ such that $uv\in E(G_2)$.
Recall that $|N_{G_2}(u_0,U)|=1$ implying $G_2[U]$ is a perfect matching of $G[U]$.
Thus, there exists a vertex $u'\in U$ such that $uu'\in E(G_2)$.
Let $L'$ be a $3$-linear forest obtained from $L$ by deleting $P$ and then adding $v'vuu'$.
Then $L'$ is a $3$-linear forest larger than $L$, a contradiction.
\end{proof}

From Claim \ref{clm-2}, if $|N_{G_2}(u_0,U)|=1$, then for each $D\in\mathcal{D}_2$, $|D|\geq 4$.

Let $A=\bigcup_{D\in \mathcal{D}_1}V(D)$ and $B=\bigcup_{D\in \mathcal{D}_2}V(D)$.
Let $B_0$ be the set of removable vertices.
Then by  Claim \ref{clm-1},
$$
N_{G_1}(u_0)=A\cup (B-B_0)\cup U'=V(G)-B_0-N_{G_2}(u_0,U)-u_0.
$$
Since $2a-1\geq |L|-\mu_L\geq \frac{2|L|}{3}$, it follows that 
$$|L|\leq 3a-2,$$ 
and hence
$$
|U'|\geq N-|L|-2\geq N-3a=n-2a\geq \frac{n}{2}.
$$

Since $N_{G_1}(u_0)=V(G)-B_0-N_{G_2}(u_0,U)-u_0$, it suffices to find a path $P_n$ in $G_1[V(G)-B_0-N_{G_2}(u_0,U)-u_0]$ obtained from two vertex disjoint paths $J,J'$, which will be defined below, by adding some edges to connect them.

We first construct $J$. Let $|\mathcal{D}_2|=t$ and $L_1,\ldots,L_t$ be all the paths in $\mathcal{D}_2$.
Then 
$$t\leq \frac{|L|}{3}\leq a-1< |U'|.$$
Let $f_i$ be the removable vertex of $L_i$.
By Claim \ref{clm-1}, each $G_1[V(L_i)]$ contains a path $L'_i$ such that $V(L'_i)=V(L_i)-f_i$.
Then $V(L'_i)\subseteq N_{G_1}(u_0)$ for each $i\in[t]$ and $\bigcup_{i\in[t]}V(L'_i)=B-B_0$.
Let $z_i^1,z_i^2$ be the endpoints of $L'_i$.
Since $t< |U'|$, we can choose a subset $Y=\{y_1,y_2,\ldots,y_t\}$ of $U'$.
Let
$$
E'=\{y_1z_1^1,y_2z_1^2,y_2z_2^1,y_3z_2^2,\ldots,y_tz_{t-1}^2,y_tz_t^1\}.
$$
Then $E'\subseteq E(G_1)$ and $J=E'\cup (\bigcup_{i\in[t]}L'_i)$ is a path of $G_1$ with endpoints $y_1$ and $z^2_t$.
It is worth noting that if $\mathcal{D}_2=\emptyset$, then let $J$ be the empty graph.

Now we construct $J'$. Let $X=U'-Y$ and $\ell=\min\{|X|, |A|\}$.
It follows from $|U'|>|Y|$ that $X\neq \emptyset$.
Let $A'$ be a subset of $A$ and $X'$ be a subset of $X$ such that $|A'|=|X'|=\ell$.
It is obvious that $G_1[A'\cup X']$ contains a red path $J'$ of order $2\ell$.
Moreover, one endpoint of $J'$ is in $A'$ (say $b'$) and the other endpoint is in $X$ (say $b$).
Since $X\neq \emptyset$, if $\ell=0$, then $A=\emptyset$, and let $J'$ be the empty graph.
%Note that if $A=\emptyset$, then let $|J'|=0$; otherwise, $|J'|\geq 2$. 
%a single vertex ($b$ exists since $t<|U'|$).

We construct a red path $P$ as follows. If both $J,J'$ are not empty, then we let $P=J\cup J'\cup\{b'y_1\}$;
if $J'$ is the empty graph, then let $P=J$;
if $J$ is the empty graph, then let $P=J'$.
It is clear that $P$ is a red path of $G_1$ with order $2\ell +|B-B_0|+t$ and $V(P)\subseteq N_{G_1}(u_0)$.
If $|X|\leq |A|$, then $\ell=|X|$ and 
$$2\ell +|B-B_0|+t=2|U'|-t+|B-B_0|\geq 2|U'|+t\geq n$$
(note that $|B-B_0|\geq 2t$ since $L$ is a $3$-linear forest).
Thus, $u_0\vee P$ contains a red $\widehat{K}_n$, a contradiction.
Hence, we assume that $|X|>|A|$ below (then $\ell=|A|$).

%
%Since there is no $P_3$ in $G_2[U'']$, it follows that $G_1[U'']$ contains a maximum path $P'$ with $|P'|\geq |U''|-1$ (say the endpoints of $P'$ are $z,z^*$).
%Moreover, $|P'|=|U''|-1$ if any only if $|U''|=2$ (say $U''=\{s_1,s_2\}$) and $s_1s_2\in E(G_2)$.

\begin{figure}[h]
    \centering
    \includegraphics[width=300pt]{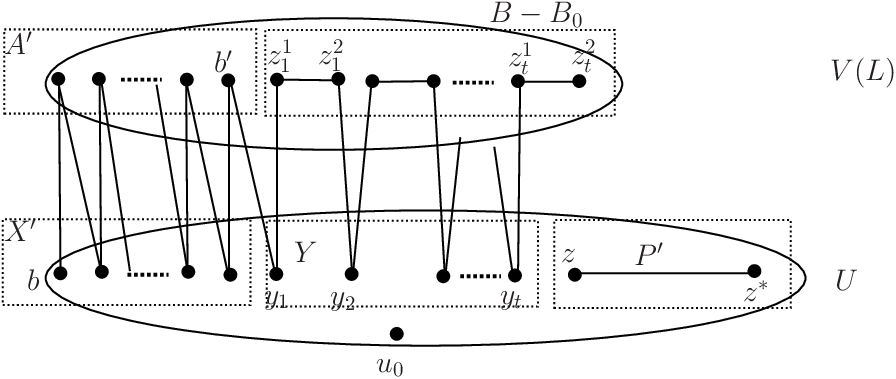}\\
    \caption{The path $P^*$.} \label{p-1}
\end{figure}

\begin{claim}\label{clm-3}
$G_1[N_{G_1}(u_0)]$ contains a path of order $N-|B_0|-1-|N_{G_2}(u_0,U)|$.
\end{claim}
\begin{proof}
Let $U''=U'-Y-X'=X-X'$. Since $X'\subseteq X$ and $|X'|=\ell=|A|<|X|$, it follows that $U''\neq \emptyset$.
Let $P'$ be a maximum path in $G_1[U'']$.
Since $\Delta(G_2[U''])\leq 1$, it follows that $P'$ is a Hamiltonian path of $G_1[U'']$ (say the endpoints of $P'$ are $z,z^*$; if $|U''|=1$, then $z=z^*$), unless $|U''|=2$ and the two vertices of $U''$ forms a blue edge.

If $P'$ is a Hamiltonian path of $G_1[U'']$, then $P^*=P'\cup P\cup \{z^2_t z\}$ is a path of $G_1$ and $|P^*|= N-|B_0|-1-|N_{G_2}(u_0,U)|$.
Otherwise, $|U''|=2$ (say $U''=\{s_1,s_2\}$) and $s_1s_2\in E(G_2)$. 
Thus, we assume that $|P'|=|U''|-1$. Then $U''=\{s_1,s_2\}$.
Note that one endpoint of $P$ is $z_t^2$ and the other endpoint is $w\in \{b,y_1\}$ (if $J'$ is the empty graph, then $w=y_1$; otherwise $w=b$).
Since $\Delta(G_2[U])\leq 1$, we have that one of $ws_1,ws_2$ is in $E(G_1)$ (by symmetry, suppose $ws_2\in E(G_1)$).
If $J$ is not the empty graph, then $P^*=P\cup\{ws_2,z_t^2s_1\}$ is a path of $G_1$, and $|P^*|=N-|B_0|-1-|N_{G_2}(u_0,U)|$;
if $J$ is the empty graph, then $P^*=P\cup\{ws_2,b's_1\}$ is a path of $G_1$, and $|P^*|=N-|B_0|-1-|N_{G_2}(u_0,U)|$.
\end{proof}

By Claim \ref{clm-3}, in order to get a $P_n$ in $G_1[N_{G_1}(u_0)]$, we only need to prove that $$N-|B_0|-1-|N_{G_2}(u_0,U)|\geq n.$$
Suppose, to the contrary, that $N-|B_0|-1-|N_{G_2}(u_0,U)|<n$.
Then 
$$n+a= N\leq n+|B_0|+|N_{G_2}(u_0,U)|,$$ and hence $a\leq |B_0|+|N_{G_2}(u_0,U)|$.
If $|N_{G_2}(u_0,U)|=0$, then $a\leq |B_0|$.
Recall that $|L|\leq 3a-2$.
Then $a\leq |B_0|\leq \mu_L\leq \frac{3a-2}{3}$, a contradiction.
If $|N_{G_2}(u_0,U)|=1$, then it follows from Claim \ref{clm-2} that each path of $\mathcal{D}_2$ has order at least $4$, and hence $|B_0|\leq \frac{|L|}{4}$.
Then $a\leq |B_0|+1\leq \frac{|L|}{4}+1\leq \frac{3a+2}{4}$, which contradicts that $a\geq 3$.
Therefore, we get a $P_n$ in $N_{G_1}(u_0)$, and hence there is a $\widehat{K}_n$ in $G_1$, a contradiction.
\end{proof}

Now we give the proof of Theorem \ref{class}.

\noindent {\bf Proof of Theorem \ref{class}:}
Let $N=n+\lceil m/2\rceil-1$ and $G=K_N$.
We partition $V(G)$ into two parts $A,B$ such that $|A|=n$ and $|B|=\lceil m/2\rceil-1$.
Let $\Gamma$ be a red/blue edge-coloring of $G$ such that each edge of $G[A]$ is red and the other edges are blue.
It is easy to verify that there is no a red copy of $\widehat{K}_n$.
Suppose that $L$ is a blue linear forest with maximum number
of edges. Since each edge of $L$ incidents to at least one vertex of $B$ and each vertex of $B$
incidents to at most two edges of $L$,  it follows that $e(L)\leq 2|B|\leq m-1$, and hence $G$ does not contain blue linear forests of size at least $m$.
Therefore, $r(\widehat{K}_n,\mathcal{L})\geq n+ \lceil\frac{m}{2}\rceil$ and $r(\widehat{K}_n,\mathcal{L}')\geq n+ \lceil\frac{m}{2}\rceil$.
By Lemmas \ref{clm-0} and \ref{lem-1}, we have $r(\widehat{K}_n,\mathcal{L})\leq n+ \lceil\frac{m}{2}\rceil$ and $r(\widehat{K}_n,\mathcal{L}')\leq n+ \lceil\frac{m}{2}\rceil$ by assuming that $a=\lceil\frac{m}{2}\rceil$.
Hence,  $r(\widehat{K}_n,\mathcal{L})=n+ \lceil\frac{m}{2}\rceil$.

\section{Proof of Theorems \ref{main-p5-paths}, \ref{main-p4+-paths} and \ref{main-k13-paths}}

In order to prove Theorems \ref{main-p5-paths}, \ref{main-p4+-paths} and \ref{main-k13-paths}, we need the exact values of $b_k(P_n)$ and $t(P_n)$ first.
The following is the exact value of $b_k(P_n)$.
\begin{lemma}\label{bkpn}
If $k\geq 3$ and $n\geq 2(k-1)$, then
$$b_k(P_n)=\left\{
\begin{array}{ll}
n, &  \mbox{if}~2(k-1)\leq n\leq 4(k-2)+1,\\
\left\lceil\frac{3n-3}{2}\right\rceil, &\mbox{if}~n>4(k-2)+1 .\\
\end{array}
\right.$$
\end{lemma}
\begin{proof}
Let $N$ be a positive integer and $G$ be an edge-colored $K_N$ such that the edge-coloring belongs to $\mathcal{B}_k(N)$. Let $V_1,\ldots,V_{k-1}$ be all the $k-1$ parts. Without loss of generality, let $2\leq |V_1|\leq \ldots\leq |V_{k-1}|$.
Let $U=\bigcup_{i=1}^{k-2}V_i$.

If $2(k-1)\leq n\leq 4(k-2)+1$, then let $N=n$.
It is obvious that $|U|\geq2(k-2)\geq \left\lfloor\frac{n}{2}\right\rfloor$.
By Lemma \ref{mutil-HC},
$G$ contains a Hamiltonian path with color 1.
Thus, $b_k(P_n)=n$.
Therefore, let $n>4(k-2)+1$ in the following discussion.

We first prove the upper bounds.
Let $N=\left\lceil\frac{3n-3}{2}\right\rceil$.
If $|U|\geq \left\lfloor\frac{n}{2}\right\rfloor$, then by Lemma \ref{mutil-HC}, $G$ contains a monochromatic path $P_n$ with color $1$.
If $|U|< \left\lfloor\frac{n}{2}\right\rfloor$, then 
\begin{align}\label{Vk-1}
|V_{k-1}|\geq N-(\left\lfloor\frac{n}{2}\right\rfloor-1)= n.
\end{align}
Thus, by Theorem \ref{path-path}, there is an integer $m\geq 2$ such that $|V_{k-1}|=r(P_n,P_m)$, and hence $G[V_{k-1}]$ contains either a monochromatic $P_m$ with color $1$ or a monochromatic $P_n$ with color $k$.
If the latter holds, then the monochromatic $P_n$ is obtained. Hence, suppose that the former holds.
If $m\geq n$, then the result follows.
Thus, let $m<n$ and $G[V_{i-1}]$ contains a monochromatic $P_m$ with color $1$.
Then $|V_{k-1}|=r(P_n,P_m)=n+\left\lfloor\frac{m}{2}\right\rfloor-1$.
Since $U\neq \emptyset$ and $V_{k-1}-V(P_m)\neq \emptyset$, it follows that $G$ contains a monochromatic $P_{m+2}$ that is obtained from the $P_m$ by adding two edges $e_1,e_2$ between $U$ and $V_{k-1}-V(P_n)$, where $e_1$ joins one endpoint of $P_m$ and a vertex of $U$ (say $x$ is the endpoint of $e_1$ belonging to $U$), and $e_2$ joins $x$ and a vertex of  $V_{k-1}-V(P_m)$.
This implies that $G$ contains a monochromatic $P_n$ when $n-2\leq m<n$.
Hence, we assume that $n-3\geq m\geq 2$ below.

\setcounter{case}{0}
\begin{case}
$m$ is odd.
\end{case}

Let $W=V_{k-1}-V(P_m)$.
If $|W|\leq |U|$, then $G[W\cup U]$ contains a monochromatic path $P_{2|W|}$ with color $1$, and hence $G$ contains a monochromatic $P_{m+2|W|}$ obtained form the monochromatic paths $P_m$ and $P_{2|W|}$ by adding an edge joins one endpoint of $P_m$ and the endpoint of $P_{2|W|}$ belonging to $U$.
Since $m+2|W|=|V_{k-1}|+|W|\geq n$, $G$ contains a monochromatic $P_n$.
If $|W|>|U|$, then $G$ contains a monochromatic copy of $P_{2|U|+m}$.
It suffices to prove that $2|U|+m\geq n$.
Since $m$ is odd, it follows that $|V_{k-1}|=n+\frac{m-1}{2}-1$, and hence $m= 2|V_{k-1}|-2n+3$. So
\begin{align*}
2|U|+m=2|U|+2|V_{k-1}|-2n+3=2N-2n+3= n.
\end{align*}

\begin{case}
$m$ is even and $m\geq 4$.
\end{case}

Let $m+2=m_1+m_2$, where $m_1,m_2\geq 3$ and $m_1,m_2$ are odd.
Since $|V_{k-1}|=n+\frac{m}{2}-1=n+\frac{(m+2)-2}{2}-1$ and $n-3\geq m$, it follows from Theorem \ref{r-linear} that
$|V_{k-1}|=r(P_n,P_{m_1}\cup P_{m_2})$, and hence $G[V_{k-1}]$ contains either a monochromatic path $P_n$ with color $k$, or a monochromatic path $P_{m_1}\cup P_{m_2}$ with color $1$. If the former holds, then $G$ contains a monochromatic $P_n$. Hence, suppose that the latter holds.
Assume that $u\in U$. Let $U'=U-\{u\}$ and $W=V_{k-1}-V(P_{m_1}\cup P_{m_2})$.
If $|W|\leq |U'|$, then $G[W\cup U']$ contains a monochromatic path $P_{2|W|}$ with color $1$.
We can choose the two endpoints of $P_{2|W|}$ as $x\in W$ and $y\in U'$.
Without loss of generality, suppose that the two endpoints of $P_{m_1}$ are $x_1',y_1'$  and the two endpoints of $P_{m_2}$ are $x_2',y_2'$. Then $P_{m_1}\cup P_{m_2}\cup P_{2|W|}\cup \{uy_1',ux_2',y_2'y\}$ is a monochromatic $P_{2|W|+(m_1+m_2)+1}$ with color $1$.  
Note that 
$$2|W|+(m_1+m_2)+1=|W|+|V_{k-1}|+1>|V_{k-1}|\geq n$$ 
by the inequality (\ref{Vk-1}). It follows that $G$ contains a monochromatic copy of $P_n$ with color $1$.
If $|W|>|U'|$, then $G$ contains a monochromatic copy of $P_{m_1+m_2+1+2|U'|}$.
We need to prove that $m_1+m_2+1+2|U'|=m+2|U|+1\geq n$.
Since $|V_{k-1}|=n+\left\lfloor\frac{m}{2}\right\rfloor-1$ and $m$ is even, it follows that $m= 2|V_{k-1}|-2n+2$, and hence
$$
m+2|U|+1=2|V_{k-1}|-2n+2|U|+3=2(N-n)+3\geq n.
$$

\begin{case}
$m$ is even and $2\leq m\leq 3$.
\end{case}

In this case, $|V_{k-1}|=n$ and $|U|=N-|V_{k-1}|=\lceil\frac{n-1}{2}\rceil-1$.
Since $|V_{k-1}|=r(P_n,P_3)$, if $G[V_{k-1}]$ does not contain a monochromatic copy of $P_n$ with color $k$, the monochromatic $P_n$ is obtained. Otherwise, $G[V_{k-1}]$ contain monochromatic $F=P_3$ with color $1$.
We can choose a subset $W$ of $V_{k-1}-V(F)$ such that $|W|=\lceil\frac{n-1}{2}\rceil-1$, since $|V_{k-1}-V(F)|= n-3\geq \lceil\frac{n-1}{2}\rceil-1$.
Then $G[W\cup U]$ contains a monochromatic $P_{2|W|}$ with color $1$, and there is a  monochromatic $P_{2|W|+3}$ with color $1$ that is obtained by connecting the endpoint of $P_{2|W|}$ belonging to $U$ and an endpoint of $P_3$.
Since $2|W|+3\geq n$, there is a monochromatic $P_n$ with color $1$.

Now we prove the lower bounds.
Let $M=\left\lceil\frac{3n-3}{2}\right\rceil-1$ and $G=K_M$.
We partition $V(G)$ into $k-1$ parts  $V_1,\ldots,V_{k-1}$ such that $|V_1|=|V_2|=\ldots=|V_{k-2}|=2$ and $|V_{k-1}|=\left\lceil\frac{3n-1}{2}-2(k-1)\right\rceil$.
Let $V_i=\{u_i,v_i\}$ for $i\in[k-2]$, and let $G[V_{k-1}]=H\cup \overline{H}$, where $H$ is the disjoint union of $H_1=K_{\left\lceil\frac{n-4(k-1)+1}{2}\right\rceil}$ and $H_2=K_{n-1}$.
Choose a $k$-edge-coloring $\Gamma\in \mathcal{B}_k(M)$ such that $u_iv_i$ is colored by $i+1$ for $i\in [k-2]$ and $H$ is a monochromatic graph with color $k$, and the other edges are colored by $1$. 
It is easy to verify that there is no monochromatic $P_n$ with color $i\geq 2$.
For color $1$, 
it is clear that the edges with color $1$ induces a spanning complete $k$-partite 
graph with $k$ parts $V_1,V_2,\ldots,V_{k-2},V(H_1),V(H_2)$.
Let $P$ be a minimum monochromatic path.
Since $V(H_2)$ is an independent set in the complete $k$-partite graph, each edge of $P$ incidents with a vertex of $V(G)-V(H_2)$, and hence 
$e(P)\leq 2|V(G)-V(H_2)|=2(M-n+1)\leq n-2.$
Thus, $|P|\leq n-1$, and there is no monochromatic $P_n$ with color $1$.
Therefore, there is no monochromatic $P_n$ in $G$, which implies that $b_k(P_n)\geq M+1=\left\lceil\frac{3n-3}{2}\right\rceil$.
\end{proof}

The following is the exact value of $t(P_n)$.

\begin{lemma}\label{tpn}
If $n\geq 3$, then
$$t(P_n)=\left\{
\begin{array}{ll}
\frac{3n}{2}-1, &  n\mbox{ is even},\\
\frac{3n-1}{2}, & n\mbox{ is odd }.
\end{array}
\right.$$
\end{lemma}
\begin{proof}
We first prove the lower bounds.
Suppose that $M=\frac{3n}{2}-2$ if $n$ is even, and $M=\frac{3n-1}{2}-1$ if $n$ is odd.
Let $\Gamma'\in \mathcal{T}(M)$ be an edge-coloring of $G'=K_M$ and the three parts are $V_1,V_2,V_3$.
If $n$ is even, then let $|V_1|=\frac{n}{2}$ and $|V_2|=|V_3|=\frac{n}{2}-1$.
If $n$ is odd, then let $|V_1|=|V_2|=|V_3|=\frac{n-1}{2}$.
Moreover, let $G'[V_i]$ be a complete graph with color $i$.
Then $K_M$ does not contain a monochromatic copy of $P_n$.
Therefore, the lower bounds are obtained.

Now we prove the upper bounds.
Suppose that $N=\frac{3n}{2}-1$ if $n$ is even, and $N=\frac{3n-1}{2}$ if $n$ is odd.
Let $\Gamma\in \mathcal{T}(N)$ be an edge-coloring of $K_N$ with the three parts $V_1,V_2,V_3$.
Without loss of generality, suppose that $|V_1|\geq |V_2|\geq |V_3|$.
Note that $V_2,V_3\neq \emptyset$.

Let $a_1=n-2|V_2|$.
If $a_1<2$, then $|V_2|\geq \lceil\frac{n-1}{2}\rceil$.
If $|V_1\cup V_2|\geq n$, then there is a monochromatic $P_n$ between $V_1$ and $V_2$.
Otherwise, suppose that $|V_1\cup V_2|\leq n-1$.
Then $|V_1|\leq \lfloor\frac{n-1}{2}\rfloor$.
Since $|V_1|\geq|V_2|\geq |V_3|$, it follows that $n$ is odd and $|V_1|=|V_2|=\frac{n-1}{2}$.
However, $|V_3|=N-|V_1\cup V_2|\geq (3n-1)/2-(n-1)=(n+1)/2>|V_2|$, a contradiction.
Therefore, suppose $a_1\geq 2$ below.

Let $a_2=|V_1|-\lfloor\frac{a_1}{2}\rfloor+1$.
Then $a_2\geq a_1$.
Otherwise, if $a_2<a_1$, then
\begin{align*}
|V_1|&=a_2+\left\lfloor\frac{a_1}{2}\right\rfloor-1\leq a_1+\left\lfloor\frac{a_2}{2}\right\rfloor-1\leq a_1+\left\lfloor\frac{a_1-1}{2}\right\rfloor-1\\
&\leq n-2|V_2|+\left\lfloor\frac{n-2|V_2|-1}{2}\right\rfloor-1
=n-3|V_2|+\left\lfloor\frac{n-3}{2}\right\rfloor\\
&\leq n-(N-|V_1|)-|V_2|+\left\lfloor\frac{n-3}{2}\right\rfloor\\
&\leq|V_1|-|V_2|<|V_1|,
\end{align*}
a contradiction.
Thus, $2\leq a_1\leq a_2$. By Theorem \ref{path-path}, we have that $|V_1|=r(P_{a_1},P_{a_2})$.
Hence, $G[V_1]$ contains either a monochromatic copy of $P_{a_1}$ with color $1$ or  a monochromatic copy of $P_{a_2}$ with color $3$.

Suppose that $G[V_1]$ contains a monochromatic copy of $P_{a_1}$ with color $1$ (say $w$ is an endpoint of $P_{a_1}$).
Let $U_1=V_1-V(P_{a_1})$ and let $\ell=\min\{|U_1|,|V_2|\}$.
Then there is a monochromatic $P_{2\ell}$ between $U_1$ and $V_2$ with color $1$.
Without loss of generality, suppose that $w_1$ are endpoints of the path $P_{2\ell}$ with $w_1\in V_2$.
Then $P_{2\ell}\cup P_{a_1}\cup \{ww_1\}$ monochromatic path of order $2\ell+a_1$ (see Figure \ref{tpn} (1)).
We only need to verify that $2\ell+a_1\geq n$ below.
If $|U_1|\geq |V_2|$, then $2\ell+a_1=2|V_2|+a_1=n$.
If $|U_1|<|V_2|$, then
$$2\ell+a_1=2|U_1|+a_1=2|V_1|-a_1=2(|V_1|+|V_2|)-n\geq 2\left\lceil\frac{2N}{3}\right\rceil-n=n.$$

\begin{figure}[h]
    \centering
    \includegraphics[width=330pt]{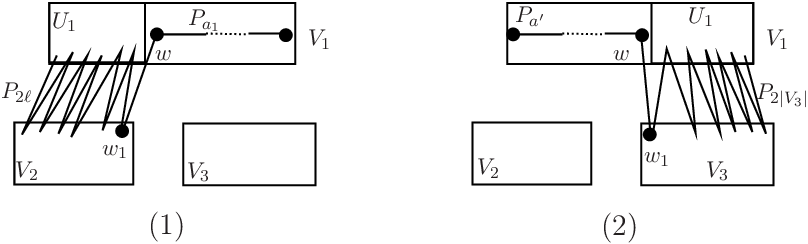}\\
    \caption{Construct a monochromatic $P_n$.} \label{tpn}
\end{figure}

Suppose that $G[V_1]$ contains a monochromatic $P_{a_2}$ with color $3$.
Let $a'=\min\{a_2,|V_1|-|V_3|\}$.
Then we can choose a subpath $P$ of $P_{a_2}$ with $|P|=a'$ (say $w$ is an endpoint of $P$).
Let $U_1=V_1-V(P)$. Then $|U_1|\geq |V_3|$.
Then there is a monochromatic $P_{2|V_3|}$ between $U_1$ and $V_3$ with color $3$.
Without loss of generality, suppose that $w_1$ is the endpoint of the path $P_{2|V_3|}$ with $w_1\in V_3$.
Then $P_{2|V_3|}\cup P_{a_2}\cup \{ww_1\}$ monochromatic path of order $2|V_3|+a'$ (see Figure \ref{tpn} (2)).
We only need to verify that $2|V_3|+a'\geq n$ below.
Note that
\begin{align*}
a_2=|V_1|-\left\lfloor\frac{a_1}{2}\right\rfloor+1=|V_1|-\left\lfloor\frac{n}{2}\right\rfloor+|V_2|+1
=(N-|V_3|)-\left\lfloor\frac{n}{2}\right\rfloor+1\geq n-|V_3|.
\end{align*}
If $a'=a_2$, then $2|V_3|+a'=2|V_3|+a_2\geq n$.
If $a'=|V_1|-|V_3|$, then 
$2|V_3|+a'=|V_1|+|V_3|\geq a_2+|V_3|\geq n$.
\end{proof}

\noindent {\bf Proof of Theorem \ref{main-p5-paths}:} Let $\Gamma$ be an edge-coloring of $K_N$ such that $K_N$ does not contain rainbow $P_5$. From Theorem \ref{P5}, one of $(i),(ii)$ holds if $k\geq 4$, and one of $(iii)-(v)$ holds if $k=4$.
Note that if $N\geq 6$, then $K_{N}$ contains a monochromatic $P_{N-1}$ with color $1$ whenever one of $(ii)-(v)$ holds.

For $2(k-1)\leq n\leq 4(k-2)+1$, let $N=n+1$.
Since $N\geq 6$, $K_{N}$ contains a monochromatic $P_n$ with color $1$ whenever one of $(ii)-(v)$ holds.
By Lemma \ref{bkpn}, $b_k(P_n)=n$. So, $\operatorname{gr}_k(P_5:P_n)\leq n+1$.
Since we can choose a $\Gamma$ of $(ii)$ such that $K_n$ does not contain monochromatic $P_n$, $\operatorname{gr}_k(P_5:P_n)=n+1$.

For $n>4(k-2)+1$, let $N=\lceil\frac{3n-3}{2}\rceil$. By Lemma \ref{bkpn},
$b_k(P_n)=N$.
Since $n\geq 10$, it follows that $N\geq n+1$, and hence $K_N$ contains a monochromatic copy of $P_n$ whenever $\Gamma$ satisfies one of $(ii)-(v)$.
Therefore, $\operatorname{gr}_k(P_5:P_n)=N=\lceil\frac{3n-3}{2}\rceil$.

\noindent {\bf Proof of Theorem \ref{main-p4+-paths}:}
Let $\Gamma'$ be an edge-coloring of $K_{N'}$ such that $K_{N'}$ does not contain rainbow $P_4^+$.
By Theorem \ref{P4+}, if $k=4$, then either $\Gamma'\in \{G_2(N'),G_3(N')\}$ or $(i)$ holds;
if $k\geq 5$, then $(i)$ holds.

Suppose $k\geq 5$. Then $\operatorname{gr}_4(P_4^+:P_n)=b_k(n)$. By Lemma \ref{bkpn},  the result holds.
Thus, we talk about the case $k=4$ below. If $n\geq 6$, then the maximum monochromatic path in $K_{N'}$ is $P_{N'-2}$ when $\Gamma'=G_2(N')$, and is $P_{N'}$ when $\Gamma'=G_3(N')$.
Thus, $\operatorname{gr}_4(P_4^+:P_n)\geq n+2$ when $n\geq 6$.
If $2(k-1)\leq n\leq 4(k-2)+1$, then $b_k(n)=n$, and hence $\operatorname{gr}_4(P_4^+:P_n)= n+2$.
If $n>4(k-2)+1=9$, then $b_k(P_n)=\lceil\frac{3n-3}{2}\rceil\geq n+2$, and hence $\operatorname{gr}_4(P_4^+:P_n)=\lceil\frac{3n-3}{2}\rceil$.

\noindent {\bf Proof of Theorem \ref{main-k13-paths}:}
Let $\Gamma^*$ be an edge-coloring of $K_{N^*}$ such that $K_{N^*}$ does not contain rainbow $K_{1,3}$.
If $k=3$, then either $\Gamma^*\in \mathcal{T}(N^*)$ or $\Gamma^*\in \mathcal{B}_3(N^*)$.
Therefore, $\operatorname{gr}_3(K_{1,3}:P_n)=\max\{t(P_n),b_3(P_n)\}=t(P_n)$.
If $k\geq 4$, then $\Gamma^*\in \mathcal{B}_k(N^*)$ and hence $\operatorname{gr}_k(K_{1,3}:P_n)=b_k(P_n)$.

\section{Proof of Theorem \ref{main-kipas}}\label{sec-kipas}

In order to prove Theorem \ref{main-kipas}, we need the exact values of $b_3(\widehat{K}_n)$ and a upper bound of $t(\widehat{K}_n)$.
The following is the exactly value of $b_3(\widehat{K}_n)$.

\begin{lemma}\label{kipas-1}
If $n\geq 5$ and $n$ is odd, then $b_3(\widehat{K}_n)=\left\lfloor\frac{5n}{2}\right\rfloor$;
if $n\geq 5$ and $n$ is even, then $\left\lfloor\frac{5n}{2}\right\rfloor-1\leq b_3(\widehat{K}_n)\leq\left\lfloor\frac{5n}{2}\right\rfloor$.
\end{lemma}
\begin{proof}
The following claim indicates the lower bounds of $b_3(\widehat{K}_n)$.
\begin{claim}
Let $N'=\lfloor(5n)/2\rfloor-1$ if $n$ is odd, and let $N'=\lfloor(5n)/2\rfloor-2$ if $n$ is even.
There exists an $3$-edge-coloring $\Gamma\in \mathcal{B}_3(n)$ of $G=K_{N'}$ such that $G$ does not contain monochromatic  $\widehat{K}_n$.
\end{claim}
\begin{proof}
We construct $\Gamma$ as follows. 
\begin{itemize}
\item Partition $V(G)$ into four parts $A,B_1,B_2,B_3$ such that $|A|=n$, $|B_1|=|B_2|=|B_3|=\frac{n-1}{2}$ if $n$ is odd, and $|B_1|=\frac{n}{2}$ and $|B_2|=|B_3|=\frac{n}{2}-1$ if $n$ is even.

\item Let $B=B_1\cup B_2\cup B_3$. Each edge of $G[A]$ is colored by $3$, each edge of $E_G[A,B]$ is colored by $1$, each edge between the three parts $B_1,B_2,B_3$ is colored $2$, and each edge in $B_1$, $B_2$ and $B_3$ are colored by $1$.
\end{itemize}
Clearly, this edge-coloring $\Gamma$ belongs to $\mathcal{B}_3(M)$ and the two parts are $A$ and $B$.
It is easy to verify that there is no monochromatic  $\widehat{K}_n$ with color $2$ or $3$.
Now we prove that there is no monochromatic $\widehat{K}_n$ with color $1$.
Suppose to the contrary that there is a monochromatic $\widehat{K}_n=v\vee P$ with color $1$, where $P$ is a monochromatic $P_n$ with color $1$ and $v$ is the center of the $\widehat{K}_n$.
Let $G_1$ be the subgraph of $G$ induced by all edges with color $1$.
Since $G_1[B]$ does not contain monochromatic $P_n$, $v\notin A$.
Thus, $v\in B_i$ for some $i\in[3]$. Then $P$ is a subgraph $G_1[(B_i-v)\cup A]$.
However, since $G_1[(B_i-v)\cup A]$ is a complete bipartite graph with two parts $A$ and $B_i-v$, and since $|B_i-v|\leq \frac{n-1}{2}-1$, it follows that the maximum path in $G_1[(B_i-v)\cup A]$ is $2|B_i-v|+1\leq n-2$, a contradiction.
\end{proof}

Now we prove the upper bound.
Suppose that $N=\left\lfloor\frac{5n}{2}\right\rfloor$. Let $G$ be an edge-colored $K_N$ and the edge-coloring belong to $\mathcal{B}_3(N)$. Then $V(G)$ can be partitioned into two parts $A,B$ with $2\leq |A|\leq |B|$ such that each edge of $G[A]$ is colored with $1$ or $3$, each edge of $G[B]$ is colored with $1$ or $2$, and each edge between $A$ and $B$ is colored with $1$.

%\begin{enumerate}
%\setcounter{case}{0}
%\begin{case}\label{case-1}
{\bf Case 1} $|A|<\frac{n}{2}$.\label{case-1}
%\end{case}
%\setcounter{subcase}{0}

In this case, $|B|\geq 2n=r(P_n,\widehat{K}_n)$ by Theorem \ref{path-kipath-1}.
Then there is either a monochromatic copy of $P_n$ with color $1$ or a monochromatic copy of $\widehat{K}_n$ with color $2$ in $G[B]$, and hence there is either a monochromatic copy of $\widehat{K}_n$ with color $2$, or a monochromatic copy of $\widehat{K}_n$ with color $1$ which is obtained form the monochromatic copy of $P_n$ with color $1$ and a vertex $v\in A$ by adding all possible edges between $v$ and $V(P_n)$.

%\setcounter{case}{1}
%\begin{case}\label{case-2}
{\bf Case 2} $\lceil\frac{n}{2}\rceil\leq |A|\leq n$. \label{case-2}
%\end{case}
%\setcounter{subcase}{1}

Clearly, $|B|\geq \left\lfloor\frac{3n}{2}\right\rfloor\geq r(K_{1,\left\lfloor n/2\right\rfloor},\widehat{K}_n)$.
Thus, there is either a monochromatic copy of $K_{1,\left\lfloor n/2\right\rfloor}$ with color $1$ or a monochromatic copy of $\widehat{K}_n$ with color $2$ in $G[B]$.
If the latter holds, then the result follows.
If the former holds, then let $u_0$ be the center and $U$ be the set of leaves of the star $K_{1,\left\lfloor n/2\right\rfloor}$.
Since $|U|=\left\lfloor n/2\right\rfloor$, $|A|\geq \lceil\frac{n}{2}\rceil$ and each edge between $U$ and $A$ is colored by $1$, it follows that there is a monochromatic path $P=P_n$ with color $1$ and $V(P)\subseteq U\cup A$, and hence $u_0\vee P$ is a monochromatic copy of $\widehat{K}_n$ with color $1$.

%\setcounter{case}{2}
%\begin{case}\label{case-3}
{\bf Case 3} $|A|>n$.\label{case-3}
%\end{case}
%\setcounter{subcase}{2}

Clearly, $|B|>n$ also holds.
Suppose that there is no monochromatic $\widehat{K}_n$ with color $2$ or $3$.
We will show that there is a monochromatic $\widehat{K}_n$ with color $1$ below.
Let $|A|=n+a$ and $|B|=n+b$, where $a,b\geq 1$.
Since $a+b= \left\lfloor\frac{n}{2}\right\rfloor$, it follows that $a,b<\frac{n}{2}$.
Since $|A|\leq |B|$, we have that $a\leq b$ and $a\leq \left\lfloor\frac{n}{4}\right\rfloor$.
By Theorem \ref{star-kipath}, $|B|\geq r(K_{1,b}, \widehat{K}_n)$ and hence there is a monochromatic copy of $K_{1,b}$ with color $1$ in $G[B]$.
Let $u_0$ be the center and let $U$ be the set of leaves in $K_{1,b}$.
Then $|U|=b$.

\begin{figure}[h]
    \centering
    \includegraphics[width=150pt]{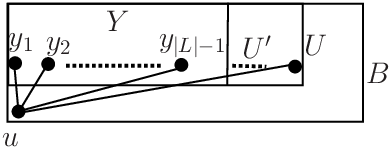}\\
    \caption{The subsets $U'$ and $Y$ of $B$.} \label{Fig-B}
\end{figure}

By Lemma \ref{clm-0} and Lemma \ref{lem-1}, one of the following statements holds.
\begin{itemize}
  \item $a\geq 3$ and there is a $3$-linear forest $L$ in $G[A]$ with color $1$ such that $|L|-\mu_L\geq 2a$, where $\mu_L$ is the number of components of $L$.
  \item $a=1$ and there is either a $2P_2$ or a $P_3$ in $G[A]$ with color $1$. 
  \item $a=2$ and there is either a $P_2\cup P_4$ (the vertex disjoint union of $P_2$ and $P_4$), or a $P_5$, or a $2P_3$ in $G[A]$ with color $1$.
\end{itemize}
In above cases, $|L|-\mu_L\geq 2a$.
If $a\geq 3$, then we can choose the $3$-linear forest $L$ such that $|L|-\mu_L=2a+\varepsilon$, where $0\leq \varepsilon\leq 2$.
\begin{claim}\label{clm-last}
$|L|\leq 3a+3$ when $a\geq 3$ and $|L|\leq 2a+2$ when $a\in\{1,2\}$.
Moreover, $\mu_L\leq |U|+1$.
\end{claim}
\begin{proof}
If $a\geq 3$, then $L$ is a $3$-linear forest and hence $\mu_L\leq |L|/3$.
Thus, $|L|\leq 3a+3$. Since $a\leq b$, it follows that 
$$\mu_L\leq \frac{|L|}{3}\leq a+1\leq b+1=|U|+1.$$
If $a=1$ or $a=2$, it is easy to verify that $|L|\leq 2a+2$ and $\mu_L\leq 2$. Hence, $\mu_L\leq 2\leq |U|+1$, since $|U|=b\geq 1$.
\end{proof}
By Claim \ref{clm-last}, we can choose $\mu_L-1$ vertices of $U$ (say $y_1,\ldots,y_{\mu_L-1}$).
Let $U'=U-Y$, where $Y=\{y_1,\ldots,y_{\mu_L-1}\}$.

\begin{figure}[h]
    \centering
    \includegraphics[width=180pt]{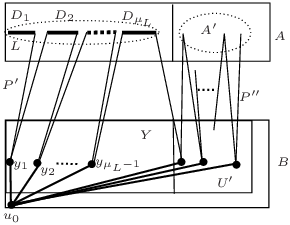}\\
    \caption{The construction of $P$.} \label{p-2}
\end{figure}

\begin{claim}\label{fact-5-1}
$|A|-|L|> |U'|$.
\end{claim}
\begin{proof}
If $a\geq 3$, then $|L|\leq 3a+3$ and
\begin{align*}
|A|-|L|\geq n+a-(3a+3)=n-2a-3\geq n-2\left(\lfloor n/2\rfloor-b\right)-3\geq 2b-3.
\end{align*}
Since $b\geq a\geq 3$, $|A|-|L|\geq b>|U'|$.
If $a\in\{1,2\}$, then 
$|L|\leq 2a+2$ and 
$$|A|-|L|\geq n+a-(2a+2)=n-a-2=n-\left(\lfloor n/2\rfloor-b\right)-2\geq b>|U'|.$$
\end{proof}
By Claim \ref{fact-5-1}, we can choose a subset $A'$ of $A-V(L)$ such that $|A'|=|U'|$.
We assume that $L$ consists of $\mu_L$ paths $D_1,D_2,\ldots,D_{\mu_L}$, and the two endpoints of $D_i$ are $z_i^1$ and $z_i^2$.
Observe that we can construct a monochromatic path $P'$ that is obtained from $L$ and $Y$ by connecting $y_i$ and $z_i^2,z_{i+1}^1$, where $i\in[\mu_L-1]$.
On the other hand, 
since the edges between $U'$ and $A'$ are colors by $1$, we can get a monochromatic $P''$ of order $2|U'|$ with one endpoint in $A'$ and the other endpoint in $U'$ (say $z$ is the endpoint of $P''$ belonging to $U'$).   
Then there is a monochromatic path $P$ with color $1$  that is obtained from $P'$ and $P''$ by adding the edge $zz_{\mu_L}^2$. Hence, $u_0\vee P$ is monochromatic kipas (see Figure \ref{p-2}).
Note that
\begin{align*}
|P|&=|L|+|Y|+|U'|+|A'|=|L|+|Y|+2(b-|Y|)\\
&=2b+|L|-|Y|=2b+|L|-\mu_L+1.
\end{align*}
If $a\geq 3$, then $|L|-\mu_L=2a+\varepsilon$.
Hence, $|P|=2b+2a+\varepsilon+1\geq n$, and $u_0\vee P$ contains a monochromatic $\widehat{K}_n$ with color $1$.
If $a\in\{1,2\}$, it is easy to check that $|L|-\mu_L= 2a$.
Hence, $|P|=2b+2a+1\geq n$.
\end{proof}

The following is an upper bound of $t(\widehat{K}_n)$.

\begin{lemma}\label{kipas-2}
If $n\geq 5$, then $t(\widehat{K}_n)\leq \left\lfloor\frac{5n}{2}\right\rfloor$ if $n$ is odd, and $t(\widehat{K}_n)\leq \left\lfloor\frac{5n}{2}\right\rfloor-1$ if $n$ is even.
\end{lemma}
\begin{proof}
Suppose that $N=\left\lfloor\frac{5n}{2}\right\rfloor$ if $n$ is odd, and $N=\left\lfloor\frac{5n}{2}\right\rfloor-1$ if $n$ is even. Let $G=K_N$.
Choose an edge-coloring of $\mathcal{T}(G)$ with parts $A,B,C$.
Without loss of generality, let $|A|\geq |B|\geq |C|$ and let $|A|=a, |B|=b$ and $|C|=c$.
Then $a\geq \lceil\frac{5n}{6}\rceil-1$.

If $|A|\geq \left\lfloor\frac{3n}{2}\right\rfloor$, then $a\geq r(P_n,P_n)$, and hence there is a monochromatic copy of $P_n$ in $G[A]$ colored by $1$ or $3$.
Since $B,C$ are both nonempty sets, it follows that there is a monochromatic copy of $\widehat{K}_n$ in $G$.
Thus, suppose that $|A|<\left\lfloor\frac{3n}{2}\right\rfloor$.
Then $b+c\geq n$ and $|B|=b\geq (b+c)/2\geq \frac{n}{2}$.
The proof will be given by contradiction, that is, $G$ does not contain a monochromatic copy of $\widehat{K}_n$.
Suppose that $H_1,H_2$ are two subgraphs induced by edges in $G[B]$ and $G[C]$ with color $2$, respectively.
Let $S$ be a star in $H_2$ such that $e(S)$ is maximum. Let $u_0$ be the center of $S$. Let $P^*$ be a path in $H_1$ such that $e(P^*)$ is maximum.
Let $X=V(S)-\{u_0\}$, $|X|=c_1$ and $|P^*|=b_1$.
Then $0\leq c_1\leq c-1$ and $1\leq b_1\leq b$.

\begin{claim}\label{k13-clm-1}
 $G[B]$ and $G[C]$ contain a monochromatic copies of $K_{1,b-b_1}$ and $K_{1,c-c_1-1}$ with color $1,3$, respectively.
\end{claim}
\begin{proof}
If $b_1<b$ and $1\leq c_1\leq c-3$, then $b\geq r(P_{b_1+1},K_{1,b-b_1})$ by Theorem \ref{star-path},  and $c\geq r(K_{1,c_1+1},K_{1,c-c_1-1})$ by Theorem \ref{star-star}.
For $c_1=0$, since $K_{1,c_1+1}$ is an edge, we also have that  $c\geq r(K_{1,c_1+1},K_{1,c-c_1-1})$.
Since $G[B]$ does not contain a monochromatic copy of $P_{b_1+1}$ and $G[C]$ does not contain a monochromatic copy of $K_{1,c_1+1}$ with color $2$ by the choice of $P^*$ and $S$, it follows that $G[B]$ and $G[C]$ contain a monochromatic copies of $K_{1,b-b_1}$ and $K_{1,c-c_1-1}$ with color $1,3$, respectively.
If $b_1=b$ or $c-2\leq c_1\leq c-1$, then $G[B]$ and $G[C]$ contain a monochromatic copies of $K_{1,b-b_1}$ and $K_{1,c-c_1-1}$ with color $1,3$, respectively.
So, $G[B]$ and $G[C]$ contain a monochromatic copies of $K_{1,b-b_1}$ and $K_{1,c-c_1-1}$ with color $1,3$, respectively.
\end{proof}

We use $S^{B,1}$ and $S^{C,3}$ to denote monochromatic copies of $K_{1,b-b_1},K_{1,c-c_1-1}$ with color $1,3$, respectively.
Moreover, suppose that the center of $S^{B,1}$ and $S^{C,3}$ are $u_1,u_2$, respectively.

\begin{claim}\label{k13-clm-2}
$b-b_1< \lfloor\frac{n}{2}\rfloor$ and $c-c_1-1<\lfloor\frac{n}{2}\rfloor$.
Moreover, if $b\geq \frac{n}{2}$, then $c_1<\lfloor\frac{n}{2}\rfloor$. 
\end{claim}
\begin{proof}
If $b-b_1\geq \lfloor\frac{n}{2}\rfloor$, then  since $|A|\geq \frac{n}{2}$, there is a monochromatic $P_n$ with color $1$ between $V(S)-u_1$ and $A$, and hence $G$ contains a monochromatic copy of $\widehat{K}_n$ with color $1$, a contradiction.
As the same reason, we have $c-c_1-1<\lfloor\frac{n}{2}\rfloor$; otherwise, there is a monochromatic $\widehat{K}_n$ with color $3$, a contradiction.
Moreover, if $b\geq \frac{n}{2}$, then $c_1<\lfloor\frac{n}{2}\rfloor$; otherwise, there is a monochromatic $\widehat{K}_n$ with color $2$, a contradiction.
\end{proof}

Let $a_1=n-2b+2b_1$ and $a_2=n-2c+2c_1+2$.
Then by Claim \ref{k13-clm-2}, $a_1,a_2\geq 2$.
We have the following claim.

\begin{claim}\label{fact-00}
There is no monochromatic $P_{a_2}$ with color $3$ in $G[A]$.
\end{claim}
\begin{proof}
We first prove that $a-n+c-c_1-1\geq 0$. If $a\geq n$, then the result follows.
If $b<\frac{n}{2}$, then $b+c_1\leq b+c-1\leq 2b-1\leq n-2$ and hence 
\begin{align*}
a-n+c-c_1-1&\geq\left(\left\lfloor\frac{5n}{2}\right\rfloor-1-b\right)-n-c_1-1\\
&= \left\lfloor\frac{3n}{2}\right\rfloor-(b+c_1)-2\\
&\geq \left\lfloor\frac{3n}{2}\right\rfloor-n\geq 0.
\end{align*}
For the case $a<n$ and $b\geq\frac{n}{2}$, by Claim \ref{k13-clm-2}, $c_1\leq\lfloor\frac{n}{2}\rfloor-1$.
Moreover, $\frac{n}{2}\leq b<n$ since $a\geq b$.
Let $b=n-q$. Then $q\geq1$.
So,
$$a-n+c-c_1-1\geq (\lfloor 5n/2\rfloor-1-b)-n-c_1-1 =\lfloor n/2\rfloor +q-c_1-2\geq 0,$$ as desired.

Assume, to the contrary, that $G[A]$ contains a monochromatic $P=P_{a_2}$ with color $3$  (say one endpoint of $P$ is $w$).
Then 
$$|A-V(P)|\geq a-n+2c-2c_1-2=(a-n+c-c_1-1)+(c-c_1-1)\geq c-c_1-1.$$
We can choose a subset $A'$ of $A-V(P)$ with $|A'|=c-c_1-1$.
Since $|A'|=|V(S^{C,3})-u_2|$, there is a monochromatic $P'$  of order $2|A'|$  between $A'$ and $V(S^{C,3})-u_2$ (say one endpoint of $P'$ is $w'$).
Then $P''=P\cup P'\cup \{ww'\}$ is a monochromatic $P_n$ with color $3$, and hence $u_2\vee P''$ a monochromatic $\widehat{K}_n$ with color $3$, a contradiction.
\end{proof}

{\bf Case 1} $c_1+b\geq n$.

The following claim is useful for the calculation.
\begin{claim}\label{fact-1}
$b_1+2c_1\leq n-1$.
\end{claim}
\begin{proof}
Suppose, to the contrary, that $b_1+2c_1\geq n$.
Then $n-2c_1\leq b_1$.
We can choose a subpath $P'$ of $P^*$ such that $|P'|=n-2c_1$ (say one endpoint of $P'$ is $a$).
Since $c_1+b\geq n$, it follows that $|B-V(P')|=b-(n-2c_1)\geq c_1$.
Choose a subset $Y$ of $B-V(P')$ such that $|Y|=|X|=c_1$ (recall that $X$ is the set of leaves of the maximum star $S$ in $H_2$).
Then there is a monochromatic path $P''$ with color $2$ between $X$ and $Y$ (say $x\in X$ is an endpoint of $P''$), and so $u_0\vee (P'\cup P''\cup\{ax\})$ is a monochromatic copy of $\widehat{K}_n$ with color $2$, a contradiction.
\end{proof}

We have proved that $G[A]$ does not contain a monochromatic copy of $P_{a_2}$ with color $3$ in Claim \ref{fact-00}.
In this case, we have the following claim.

\begin{claim}\label{fact-0}
There is no monochromatic copy of $P_{a_1}$ with color $1$ in $G[A]$.
\end{claim}
\begin{proof}
Assume, to the contrary, that $G[A]$ contains a monochromatic copy of $P=P_{a_1}$ with color $1$ (say one endpoint of $P$ is $w$).
Then $|A-V(P)|\geq a-n+2b-2b_1$.
We first prove that $a-n+b-b_1\geq 0$.
Since $c-c_1-1<\lfloor\frac{n}{2}\rfloor$ by Claim \ref{k13-clm-2}, it follows that $c\leq\lfloor\frac{n}{2}\rfloor+c_1$, and hence
\begin{align*}
a-n+b-b_1&\geq \left(\left\lfloor\frac{5n}{2}\right\rfloor-1-c\right)-n-b_1
=\left\lfloor\frac{3n}{2}\right\rfloor-(c+b_1)-1\\[0.1cm]
&\geq  \left\lfloor\frac{3n}{2}\right\rfloor-\left(\left\lfloor\frac{n}{2}\right\rfloor
+c_1+b_1\right)-1\\[0.1cm]
&=n-[(b_1+2c_1)-c_1]-1\\[0.1cm]
&\geq n-[(n-1)-c_1]-1=c_1\geq 0,
\end{align*}
as desired.
Recall that 
$$|A-V(P)|\geq a-n+2b-2b_1=(a-n+b-b_1)+(b-b_1)\geq b-b_1.$$
We can choose a subset $A'$ of $A$ such that $|A'|=b-b_1$.
Since $|A'|=|V(S^{B,1})-u_1|$, there is a monochromatic $P'$  of order $2|A'|$  between $A'$ and $V(S^{B,1})-u_1$ (say one endpoint of $P'$ is $w'$).
Then $P''=P\cup P'\cup \{ww'\}$ is a monochromatic $P_n$ with color $1$, and hence $u_1\vee P''$ a monochromatic $\widehat{K}_n$ with color $1$, a contradiction.
\end{proof}

By Claims \ref{fact-00} and \ref{fact-0}, we have $|A|=a<r(P_{a_1},P_{a_2})$.
However, if $a_1>a_2$, then
\begin{align*}
r(P_{a_1},P_{a_2})&=a_1+\left\lfloor\frac{a_2}{2}\right\rfloor-1
=\left\lfloor\frac{3n}{2}\right\rfloor-(b+c)-b+2b_1+c_1\\[0.1cm]
&\leq\left\lfloor\frac{3n}{2}\right\rfloor-\left(\left\lfloor\frac{5n}{2}\right\rfloor-1-a\right)
-b+2b_1+c_1\\[0.1cm]
&=a-n-(b-b_1)+(b_1+c_1)+1.
\end{align*}
Since $b\geq b_1$ and $b_1+2c_1<n$, we have $r(P_{a_1},P_{a_2})\leq a$, a contradiction.
If $a_1\leq a_2$, then
\begin{align*}
r(P_{a_1},P_{a_2})&=a_2+\left\lfloor\frac{a_1}{2}\right\rfloor-1
=\left\lfloor\frac{3n}{2}\right\rfloor-(b+c)-c+2c_1+b_1+1\\[0.1cm]
&\leq\left\lfloor\frac{3n}{2}\right\rfloor-\left(\left\lfloor\frac{5n}{2}\right\rfloor-1-a\right)
-c+(2c_1+b_1)+1\\[0.1cm]
&\leq (a-n)-c+(n-1)+2\leq a,
\end{align*}
a contradiction.

{\bf Case 2} $c_1+b<n$.

By Claim \ref{fact-00}, $G[A]$ does not contain monochromatic $P_{a_2}$ with color $3$.
Let 
$$a'=a-a_2+1=a-n+2c-2c_1-1.$$
Since $a\geq r(P_{a_2},K_{1,a'})$, it follows that $G[A]$ contains a monochromatic copy of $S^*=K_{1,a'}$ with color $1$.
Note that
$$
a'\geq\left(\left\lfloor\frac{5n}{2}\right\rfloor-1-b-c\right)-n+2c-2c_1-1
=\left\lfloor\frac{3n}{2}\right\rfloor-b+c-2c_1-2.
$$
Since $c_1+b<n$, it follows that $a'\geq \lfloor\frac{n}{2}\rfloor-c_1+c-1\geq \lfloor\frac{n}{2}\rfloor$.
Let $u^*$ be the center of $S^*$ and $U^*$ be the set of leaves of $S^*$.
Then $|U^*|=a'\geq \lfloor\frac{n}{2}\rfloor$.
Since $b\geq \frac{n}{2}$, there is a monochromatic $P=P_n$ with color $1$ between $A$ and $U^*$. Hence, $u^*\vee P$ is a monochromatic copy of $\widehat{K}_n$ with color $1$, a contradiction.
\end{proof}

\noindent {\bf Proof of Theorem \ref{main-kipas}:} If $n\geq 5$, then the result follows immediately from Lemmas \ref{kipas-1} and  \ref{kipas-2}.

If $n\in\{2,3\}$, then let $N=\lfloor\frac{5n}{2}\rfloor$.
Let $\Gamma\in \mathcal{B}_3(N)$ be an edge-coloring of $G=K_N$ with parts $A_1,A_2$ and $|A_1|\geq |A_2|$.
Then $|A_1|\geq 4$ when $n=3$ and $|A_1|=3$ when $n=2$.
Suppose that each edge of $G[A_i]$ is colored by $1$ or $i+1$, and the edges between $A_1$ and $A_2$ are colored by $1$. Let $H$ be the subgraph of $G[A]$ induced by edges with color $1$.
For $n=3$, if $\Delta(H)\geq 2$ or $\Delta(H)=1$ and $|A_2|\geq 2$, then there is a monochromatic copy of $\widehat{K}_3$ with color $1$; if $\Delta(H)= 0$ or $\Delta(H)= 1$ and $|A_2|=1$, then there is a monochromatic copy of $\widehat{K}_3$ with color $2$.
For $n=2$, if $\Delta(H)\geq 1$, then there is a monochromatic copy of $\widehat{K}_2$ with color $1$; if $\Delta(H)=0$, then there is a monochromatic copy of $\widehat{K}_2$ with color $2$.
Thus, we have that $b_3(\widehat{K}_n)\leq N$ for $n\in\{2,3\}$.
Let $\Gamma_1$ (resp. $\Gamma_2$) be an edge-coloring of $K_4$ (resp. $K_6$) such that the monochromatic graphs induced by color $1, 2$ and $3$ are $K_{2,2}$, $K_2$ and $K_2$ (resp. $K_{3,3}$, $K_3$ and $K_3$), respectively.
Then $\Gamma_1\in \mathcal{B}_3(4)$ (resp. $\Gamma_2\in \mathcal{B}_3(6)$) but there is not a monochromatic copy of $\widehat{K}_2$ (resp. $\widehat{K}_3$).
Therefore, we have $b_3(\widehat{K}_n)=\lfloor\frac{5n}{2}\rfloor$ for $n\in\{2,3\}$.

If $\gamma\in \mathcal{T}(N)$ is an edge-coloring of $K_N$ with the maximum part $A$, then $|A|\geq n$, and hence $G[A]$ contains a monochromatic copy of $P_n$ with color $1$ or $3$, and so $G$ contains a monochromatic copy of $\widehat{K}_n$, and hence $\operatorname{gr}_3(K_{1,3}:\widehat{K}_n)=\lfloor\frac{5n}{2}\rfloor$ for $n\in\{2,3\}$.

\section{Acknowledgements}

We gratefully thank the anonymous reviewers for their careful reading of our manuscript and their valuable comments and suggestions.
This paper is supported by the National Natural Science Foundation of China (Nos. 12201375, 12061059)


\begin{thebibliography}{1}

%\bibitem{BMOP}
%R. Bass, C. Magnant, K. Ozeki and B. Pyron. Characterizations of edge-colorings of
%complete graphs that forbid certain rainbow subgraphs. manucsript (2015).


\bibitem{BondyMurty}
J.A. Bondy, U.S.R. Murty, \emph{Graph Theory}, GTM 244, Springer,
2008.

\bibitem{CH72}
V. Chv\'{a}tal, F. Harary, Generalized Ramsey theory for graphs. III. Small off-diagonal numbers, \emph{Pacific J. Math.} 41(2) (1972), 335--345.

\bibitem{CG}
F.R.K. Chung, R.L. Graham, Edge-colored complete graphs with precisely colored
subgraphs, {\em Combinatorica} 3 (1983), 315--324.

\bibitem{EMT71}
P. Erd\H{o}s, R.J. McEliece, H. Taylor, Ramsey bounds for graph products, \emph{Pacific J. Math.} 37(1) (1971), 45--46.


\bibitem{FS}
R.J. Faudree, R.H. Schelp, Ramsey numbers for all linear forests, {\em Discrete Math.} 16 (1976), 149--155.

\bibitem{FGP}
J. Fox, A. Grinshpun, J. Pach, The Erd\"{o}s-Hajnal conjecture for rainbow triangles,
{\em J. Combin. Theory Ser. B} 111 (2015), 75--125.

\bibitem{FM-K3+}
S. Fujita, C. Magnant, Extensions of Gallai-Ramsey results, {\em J. Graph Theory} 70(4) (2012), 404--426.

\bibitem{FMO}
S. Fujita, C. Magnant, K. Ozeki, Rainbow generalizations of Ramsey theory-A dynamic survey, {\em Theory Appl. Graphs.} 0(1), Article 1, (2014)

\bibitem{Gallai}
T. Gallai, Transitiv orientierbare Graphen, {\em Acta Math. Acad. Sci. Hungar.} 18 (1967), 25--66.


\bibitem{GG}
L. Gerencs\'{e}r, A. Gy\'{a}rf\'{a}s, On Ramsey-Type Problems, {\em Annales Universitatis Scientiarum
Budapestinensis, E\"{o}tv\"{o}s Sect. Math.} 10 (1967), 167--170.


\bibitem{GLST87}
A. Gy\'{a}rf\'{a}s, J. Lehel, R.H. Schelp, Z.S. Tuza, Ramsey numbers for local colorings, \emph{Graphs Combin.} 3 (1987), 267--277.

\bibitem{GSS}
A. Gy\'{a}rf\'{a}s, G.N. S\'{a}rk\"{o}zy, A. Seb\"{o}, S. Selkow, Ramsey-type results for Gallai colorings, {\em J. Graph Theory}, 64 (2010), 233--243.


\bibitem{GS}
A. Gy\'{a}rf\'{a}s, G. Simonyi, Edge colorings of complete graphs without tricolored triangles. {\em J. Graph Theory} 46(3) (2004), 211--216.


\bibitem{Hara}
F. Harary, Recent results on generalized Ramsey theory for graphs, {\em Graph Theory and Applications,  Springer, Berlin}, 125--138 (1972)

\bibitem{LZB}
B. Li, Y. Zhang, H. Broersma, An exact formula for all star-kipas Ramsey numbers, {\em Graphs Combin.} 33 (1), 141--148.

\bibitem{LZH}
B. Li, Y. Zhang, H. Broersma, P. Holub, Closing the gap on path-kipas Ramsey numbers, {\em Electron. J. Combin.} 22(3) (2015), 3--21.

\bibitem{LMSS}
H. Liu, C. Magnant, A. Saito, I. Schiermeyer, Y.T. Shi, Gallai-Ramsey number for
$K_4$, {\em J. Graph Theory} 94 (2020), 192--205.

\bibitem{LWL}
X. Li, L. Wang, X. Liu, Complete graphs and complete bipartite graphs without rainbow path, \emph{Discrete Math.} 342 (2019), 2116--2126.

\bibitem{L}
E. Lucas, R\'{e}creations Math\'{e}matiqu\'{e}s, Vol. II. {\em Gauthier-Villars}, Paris, (1892)

\bibitem{MN-book}
C. Magnant, P.S. Nowbandegani, \emph{Topics in Gallai-Ramsey theory}, Springer Nature, Cham, Switzerland, 2020.

\bibitem{MS-K5}
C. Magnant, I. Schiermeyer, Gallai-Ramsey number for $K_5$, {\em J. Graph Theory} 101(3) (2022), 455--492.




\bibitem{Ptd}
T.D. Parsons, Path-star Ramsey numbers, {\em J. Combin. Theory, Ser. B}  17(1) (1974), 51--58.

\bibitem{Radzi}
S.P. Radziszowski, Small Ramsey numbers, {\em Electron. J. Combin.}, Dynamic Survey, 1--30 (1994)

\bibitem{SB}
A.N.M. Salman, H.J. Broersma, Path-kipas Ramsey numbers, {\em Discrete Applied
Mathematics} 155 (2007), 1878--1884

\bibitem{TW}
A. Thomason, P. Wagner, Complete graphs with no rainbow path.
{\em J. Graph Theory} 54 (3) (2007), 261--266.


\bibitem{WHMZ22}
M. Wei, C. He, Y. Mao, X. Zhou, Gallai-Ramsey numbers for rainbow $P_5$ and monochromatic fans or wheels, \emph{Discrete Math.} 345 (2022), 113092.
\end{thebibliography}
\end{document}